\documentclass[final, 12pt]{article}

\usepackage{hyperref}
\usepackage{amsmath}
\usepackage{amsfonts}
\usepackage{makeidx}
\usepackage{amsthm}
\usepackage{mathtools}
\usepackage{amssymb}
\usepackage{bm}
\usepackage{cite}
\usepackage{showkeys}
\usepackage[english]{babel}
\usepackage{dblaccnt}
\usepackage{accents}
\usepackage{graphicx}
\usepackage{psfrag}
\usepackage{subcaption}
\usepackage{color}
\usepackage{float}
\usepackage{amscd}
\usepackage[mathscr]{euscript}
\usepackage{lipsum}
\usepackage{epstopdf}
\usepackage{algorithm}
\usepackage{algpseudocode}

\graphicspath{{images/}}

\newcommand{\Z}{\ensuremath{\mathbb{Z}}}
\newcommand{\R}{\ensuremath{\mathbb{R}}}

\renewcommand{\Re}{\mathrm{Re}\,}

\newtheorem{defi}{Definition}

\newtheorem{theorem}[defi]{Theorem}

\newtheorem{remark}[defi]{Remark}

\title{A new sampling indicator function for stable imaging of  periodic scattering media }

\author{Dinh-Liem Nguyen\thanks{Department of Mathematics, Kansas State University, Manhattan, KS 66506; (\texttt{dlnguyen@ksu.edu}, \texttt{stahlkj@ksu.edu},  \texttt{trungt@ksu.edu})} \and Kale Stahl\footnotemark[1]   \and Trung Truong\footnotemark[1] 
}

\begin{document}

\date{}
\maketitle

\begin{abstract}
This paper is concerned with the inverse  problem of determining the shape of penetrable periodic scatterers from 
scattered field data.  We propose a sampling method with a novel indicator function  for solving this inverse problem. 
This indicator function is  very simple to implement and robust against noise in the data. The resolution 
 and stability  analysis of the indicator function is analyzed. Our numerical study shows that the  proposed sampling method 
 is more stable than the factorization method and more efficient than the direct or orthogonality sampling method in reconstructing periodic 
 scatterers. 
\end{abstract}

\sloppy


\bigskip


\section{Introduction}

In this paper we aim to numerically solve the inverse scattering problem for  periodic 
media in $\R^2$. The periodic media of interest are unboundedly periodic  in 
the horizontal direction and bounded in the vertical direction. These periodic media are motivated by  one-dimensional 
photonic crystals and the inverse problem of interest  
is inspired by applications of nondestructive evaluations for photonic crystals. There have 
been an increasing amount  of studies on 
numerical methods for shape reconstruction of periodic scattering media 
 during the past years, see~\cite{Arens2005, Arens2003a, Bao2014, Elsch2012, Hadda2017, 
Lechl2013b, Nguye2014, Sandf2010, Yang2012, Nguye2016, Jiang2017, Cakon2019, Nguye2022b, Nguye2022}. Two major approaches that were studied in these papers are the factorization method  and the near field imaging method. The latter method, which  relies  on a transformed field expansion,  can provide  super-resolved resolution. However,  this method requires  the periodic scattering structure to be   a smooth periodic function multiplied by a small surface deformation parameter.  The factorization method can essentially work for periodic scattering structures of arbitrary shape but it  is not very robust  against the noise in the scattering data. This method belongs to the class of sampling or qualitative methods that were introduced by D. Colton and A. Kirsch~\cite{Colto1996, Kirsc1998}. The factorization method aims to construct a necessary and sufficient characterization of the unknown scatterer from multi-static data. We refer to~\cite{Kirsc2008} for more details about the factorization method.  

In this work we develop a sampling method with a new indicator  function to solve the inverse scattering problem for periodic media. This sampling method is inspired by the  orthogonality sampling method~\cite{Potth2010} and the direct sampling method~\cite{Ito2012}. These two sampling methods share similar ideas and features and were studied independently. For  simplicity we will refer to them as the orthogonality sampling method. The computation of the new indicator function  is very simple and fast as one only needs to evaluate two finite sums involving  the propagating modes of the scattered field data. Like the  orthogonality sampling method, the proposed sampling method also does not involve solving any ill-posed problems and it is very robust against noise in the data. The resolution of the new indicator function is studied using Green's identities and the Rayleigh expansion  of the $\alpha$-quasiperiodic  fields  of the scattering problem. The stability of the indicator function is also established.  The  performance of the new indicator function is studied in various contexts in the numerical study. The numerical study also shows that the proposed sampling method is more robust than the factorization method and more efficient than the  orthogonality sampling method in reconstructing periodic scattering media.
We also want to mention that although the orthogonality sampling method has been studied for inverse scattering from bounded objects~\cite{Potth2010, Gries2011, Ito2012, Ito2013, Kang2018, Harri2020, Le2022}, its application  to the inverse scattering problem for periodic media is still not known.

The paper is organized as follows. The basics of the scattering from periodic media  and the inverse problem of interest are described in Section~\ref{setup}. The new indicator function and its resolution and stability analysis are discussed in Section~\ref{theory}.  
Section~\ref{numerical} is dedicated to a numerical study of the new indicator function and its comparison to the factorization method and the orthogonality sampling method.

\section{Problem setup}
\label{setup}

We consider a two-dimensional  medium which is unboundedly $2\pi$-periodic in $x_1$-direction and bounded in $x_2$-direction. 
Let  $n$ be a bounded function which is  $2\pi$-periodic with respect to $x_1$. Suppose that 
 the interior of the periodic medium is characterized by $n$ and that the exterior of the
 periodic medium is homogeneous which means $n = 1$ in these areas.  Note that the period can be any arbitrary value, but it is chosen to be $2\pi$ for the convenience of the presentation.
 For $\alpha \in \R$,
 we define that a function $f$ is $\alpha$-quasiperiodic in $x_1$ if  
 $$
   f(x_1 + 2\pi j,x_2) = \text{e}^{i2\pi j \alpha} f(x_1,x_2), \quad j\in \Z, \quad (x_1,x_2)^\top \in \R^2.
$$
From now on we will call functions with this property $\alpha$-quasiperiodic functions for short. 
A typical example of $\alpha$-quasiperiodic functions is a plane wave  (e.g. $\exp(i k(d_1 x_1 + d_2x_2))$ with $d_1^2+d_2^2 = 1$, $k>0$). 
Suppose that the periodic medium is illuminated by  an $\alpha$-quasiperiodic incident field $u_{in}$ with wave number $k>0$.  
Note that since the medium is unboundedly periodic in $x_1$, we are only interested  incident fields propagating downward or upward toward 
the medium.
The scattering of this incident field  by  the  periodic  medium
produces the scattered field $u_{\mathrm{sc}}$ described by
\begin{align}
\label{eq:HmodeEquation1}
\Delta u_{\mathrm{sc}} + k^2 u_{\mathrm{sc}} = -k^2 q u \quad \text{in } \R^2,
\end{align}
where $q$ is the contrast given by
\[
 q = n - 1.
\]
It is well known for this  scattering problem that the scattered field $u_{\mathrm{sc}}$ 
must also be $\alpha$-quasiperiodic, and that the direct problem of finding 
the scattered field can be reduced to one period 
$$
\Omega := (-\pi,\pi)\times \R.
$$
Let $D = \text{supp}(q) \cap \Omega$. For $h>0$ such that
\begin{align}
\label{hh}
h>\sup \big\{|x_2|: \, (x_1,x_2)^\top \in D \big\},
\end{align}
the direct scattering problem is completed by the Rayleigh 
expansion condition for the scattered field
\begin{equation}
\label{eq:radiationCondition}
  u_{sc}(x) =
  \begin{cases}
   \sum_{j\in \Z} {u}^{+}_j \text{e}^{i \alpha_j x_1 + i \beta_j (x_2- h)}, &  x_2 \geq h, \\
      \sum_{j\in \Z} {u}^{-}_j \text{e}^{i \alpha_j x_1 - i \beta_j (x_2 + h)}, &  x_2 \leq - h, 
  \end{cases}
\end{equation}
where 
\begin{align*}
   \alpha_j := \alpha + j, \quad 
\beta_j:= \begin{cases} 
              \sqrt{k^2 - \alpha_j^2}, & k^2 \geq \alpha_j^2 \\ 
              i \sqrt{\alpha_j^2- k^2}, & k^2 < \alpha_j^2
            \end{cases}, \quad  j \in \Z,
\end{align*}
and $({u}^{\pm}_j)_{j\in\Z}$ are the (complex-valued) Rayleigh sequences of the scattered field $u_{sc}$.  
%
The condition~\eqref{eq:radiationCondition} means that the scattered field $u_{{sc}}$ is 
an outgoing wave. Note that only a finite number of terms in~\eqref{eq:radiationCondition} are propagating 
plane waves which are called propagating modes, the rest are evanescent modes which correspond 
to exponentially decaying terms. From now, we call a function satisfying~\eqref{eq:radiationCondition} a
radiating function. In addition, we also assume that $\beta_j$ is nonzero for all $j$ which means  
the Wood anomalies are excluded in our analysis.


Well-posedness of scattering problem~\eqref{eq:HmodeEquation1}--\eqref{eq:radiationCondition}
is well-known, see for instance~\cite{Bonne1994}.
For $r>0$  define  
\[
  \Omega_r:=(-\pi,\pi)\times(-r, r), \quad \Gamma_{\pm r}:= (-\pi,\pi)\times\{\pm r\}.
\]
Recall the constant $h$ in~\eqref{hh}. For the inverse problem of interest we  measure the scattered field $u_{sc}$ on $\Gamma_{\pm r}$ for some 
$r \geq h$. From the Rayleigh expansion of $u_{sc}$, knowing  $u_{sc}$ on $\Gamma_{\pm r}$  is equivalent to knowing the Rayleigh coefficients $(u^\pm_j)_{j \in \Z}$. Since the evanescent modes  are associated with the exponentially decaying terms in the Rayleigh expansion~\eqref{eq:radiationCondition}, it is typically difficult to obtain these modes in practice unless one can measure extremely near  (e.g. within one wavelength)  the periodic scatterers. Thus, we consider only the propagating modes for the scattering data of our inverse problem   as follows.

\textbf{Inverse problem.} Given the Rayleigh coefficients $( {u}_j^\pm )$  for $j \in \Z$ such that $\beta_j > 0$, determine $D$.

\section{A new indicator function and its properties}
\label{theory}

Recall that the $\alpha$-quasiperiodic Green function of the direct problem is given by
\begin{equation}
 \label{eq:GreenForm1}
 G(x,y) = \frac{i}{4\pi}\sum_{j \in\Z} \frac{1}{\beta_j}\text{e}^{i \alpha_j(x_1-y_1) + i \beta_j|x_2-y_2|}, 
\quad x,y\in \Omega,\, x_2\neq y_2.
\end{equation}
It is well known that the direct problem is equivalent to 
the Lippmann-Schwinger equation
\begin{equation}\label{lse}
u_{sc}(x) = k^2\int_{D} G(x,y)q(y)u(y)\, dy.
\end{equation}
Let $N$ be the number of incident fields we use for the inverse problem. 
Let ${u}^\pm_j(l)$ be the Rayleigh coefficients of the scattered field $u_{sc}(\cdot,l)$ generated by incident field $u_{in}(\cdot,l)$ for $l = 1, 2, \dots, N$. 
For $p\in\mathbb{N}, z \in \Omega$, define the following indicator function
$$
I(z) := \sum_{l = 1}^N \left| \sum_{j:\beta_j>0}\beta_j \left({u}^+_j(l) \overline{{g}^+_j(z)} + {u}^-_j(l) \overline{{g}^-_j(z)} \right) \right|^p, 
$$
where 
$$
{g}^\pm_j(z) = \frac{i}{4\pi\beta_j}\text{e}^{-i\alpha_j z_1 \mp i\beta_j(z_2 \mp h)}.
$$
This indicator function $I(z)$ aims to determine $D$ in $\Omega$ and $z$ plays the role of sampling points. 
We note that $g^\pm_j(z)$ are also the Rayleigh coefficients of the $\alpha$-quasiperiodic Green function. 
 In the proof of following theorem we will drop the 
dependence of ${u}^\pm_j$ on $l$ for the convenience of the presentation but we keep $l$ in the total field $u(y,l)$ that is generated by 
incident field $u_{in}(y,l)$. We analyze  the behavior of  $I(z)$ in the following theorem.
\begin{theorem}
The indicator function satisfies
\begin{align*}
I(z) = \left(\frac{k^2}{8\pi} \right)^p \sum_{l = 1}^N \left| \int_{D}  \left[ J_0(k|z-y|)+ w_{\alpha}(z, y) \right]q(y)u(y,l) dy  \right|^p
\end{align*}
where  $J_0$ is the Bessel function of the first kind, and
\begin{equation}
\label{w}
w_{\alpha}(z,y) :=  \sum_{ j \in \Z\setminus\{0\} }e^{-i2\pi j \alpha}  J_0 \left(k\sqrt{(z_1-y_1 + 2j\pi)^2 + (z_2 - y_2)^2} \right).
\end{equation}
\end{theorem}

\begin{remark}
Since $z_1 - y_1 + 2j\pi$ is always nonzero for  $z_1,  y_1 \in (-\pi,\pi)$ and $j \neq 0$, the series in~\eqref{w} converges, and
$J_0 (k\sqrt{(z_1-y_1 + 2j\pi)^2 + (z_2 - y_2)^2} )$ has no peak at $z = y$. In fact, we   numerically observe that 
the kernel function
$J_0(k|z-y|)+ w_{\alpha}(z, y)$ peaks at $z = y$ and is relatively small  as $y$ and $z$ are away from each other, see Figure~\ref{fig0} for an example. 
Thus we expect from the theorem that $I(z)$  has small values as $z$ is outside $D$ and has much larger values as $z$ is inside $D$. 

%

\end{remark}

\begin{figure}[h!]
\centering
\begin{subfigure}{0.3\textwidth}
\centering
\includegraphics[trim=60 30 60 30,clip,width=\textwidth]{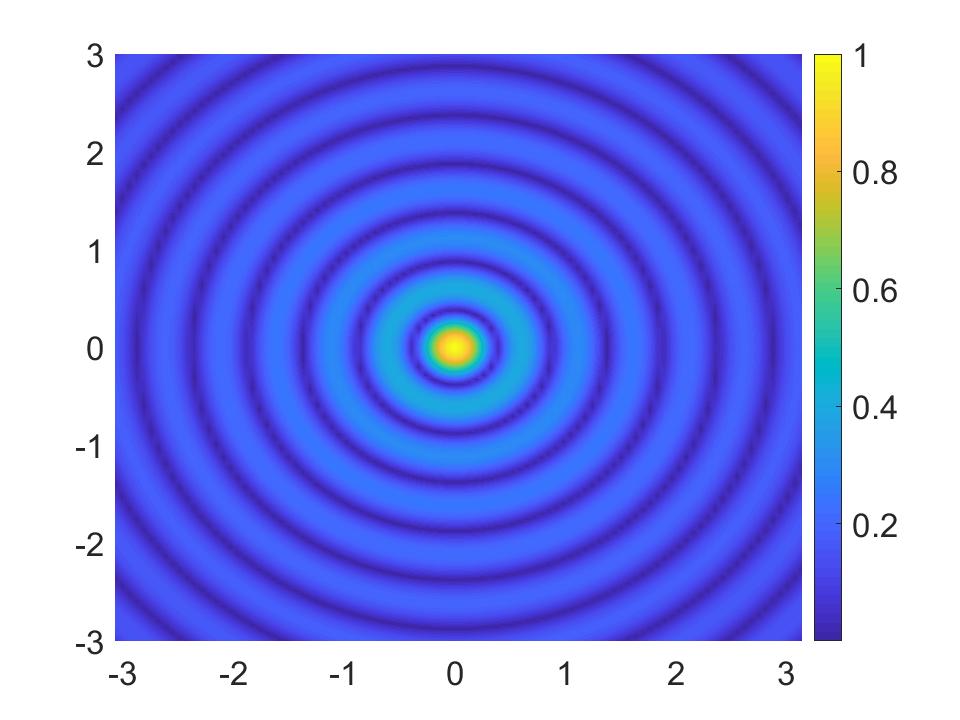}
\caption{$J_0(k|x-y|)$}
\end{subfigure}
\hspace{0.0cm}
\begin{subfigure}{0.3\textwidth}
\centering
\includegraphics[trim=60 30 60 30,clip,width=\textwidth]{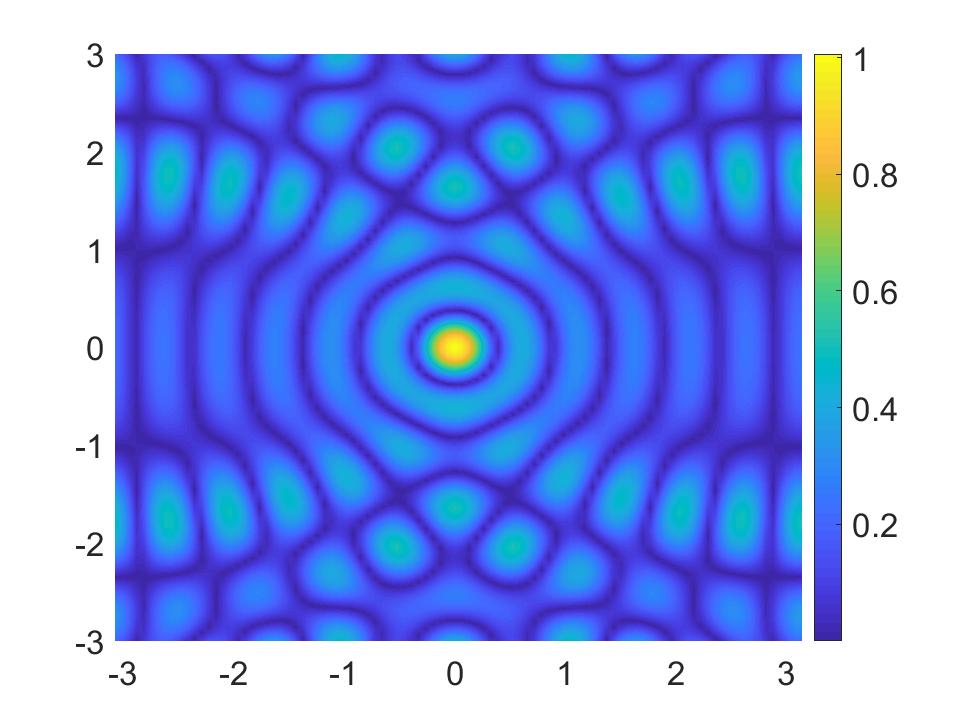}
\caption{$|J_0(k|x-y|) + w_0(x,y)|$}
\end{subfigure}
\begin{subfigure}{0.3\textwidth}
\centering
\includegraphics[trim=60 30 60 30,clip,width=\textwidth]{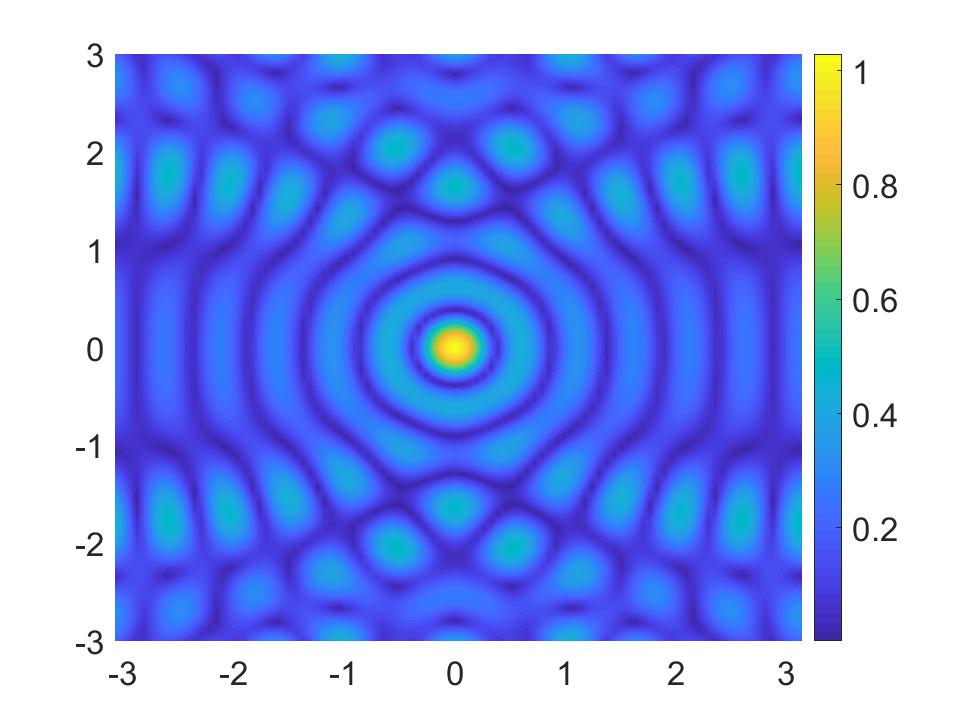}
\caption{$|J_0(k|x-y|) + w_{\pi/3}(x,y)|$}
\end{subfigure}
%
\centering
\begin{subfigure}{0.3\textwidth}
\centering
\includegraphics[trim=60 30 60 30,clip,width=\textwidth]{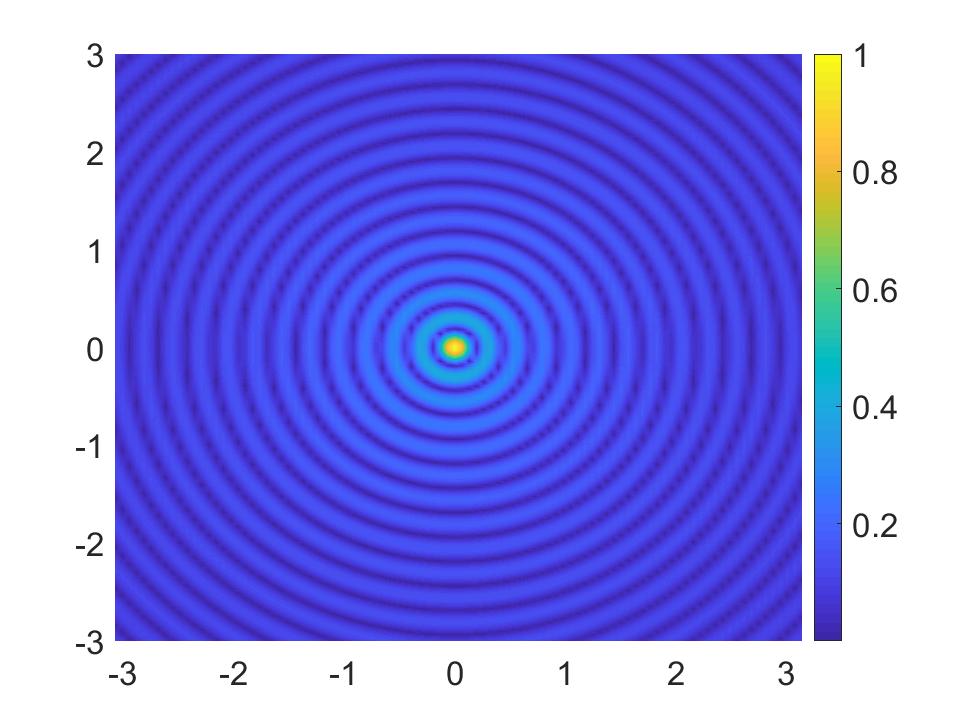}
\caption{$J_0(k|x-y|)$}
\end{subfigure}
\hspace{0.0cm}
\begin{subfigure}{0.3\textwidth}
\centering
\includegraphics[trim=60 30 60 30,clip,width=\textwidth]{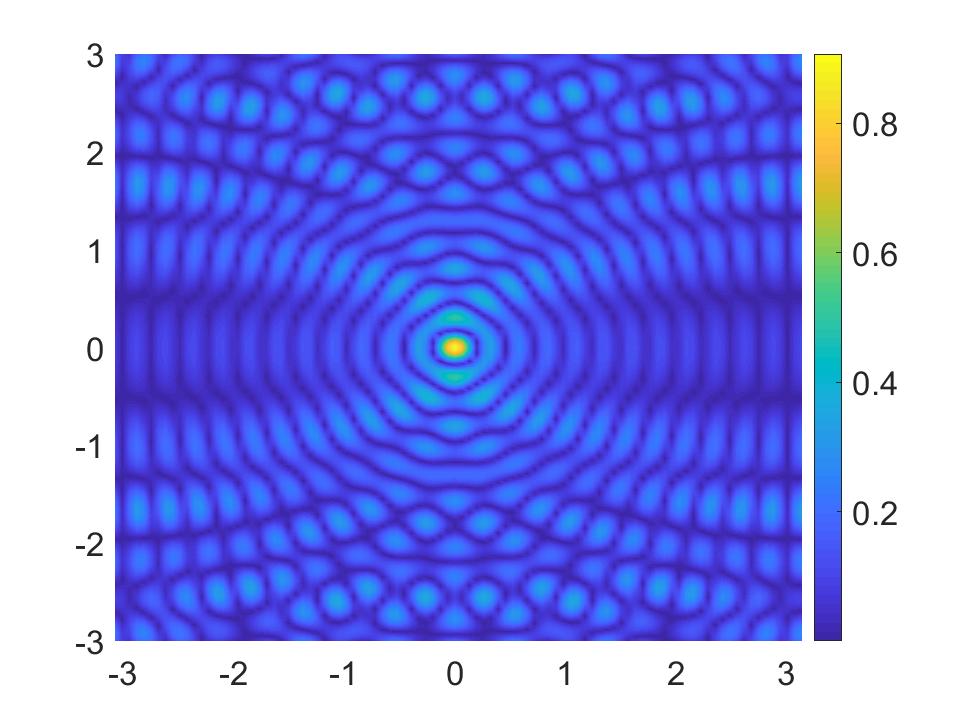}
\caption{$|J_0(k|x-y|) + w_0(x,y)|$}
\end{subfigure}
\begin{subfigure}{0.3\textwidth}
\centering
\includegraphics[trim=60 30 60 30,clip,width=\textwidth]{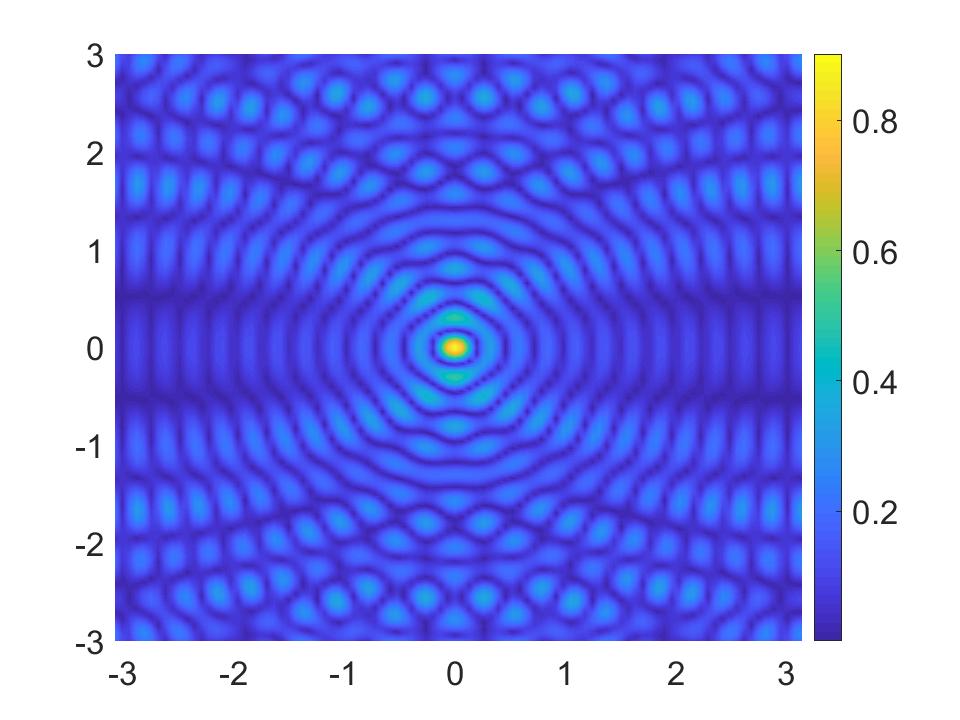}
\caption{$|J_0(k|x-y|) + w_{\pi/3}(x,y)|$}
\end{subfigure}
\caption{Functions $J_0(k|x-y|)$ and $|J_0(k|x-y|) + w_\alpha(x,y)|$ for $x \in (-\pi, \pi)\times (-3,3)$, $y=0$,  $k = 2\pi$ (first row) and $k = 4\pi$ (second row).}
\label{fig0}
\end{figure}


\begin{proof}
For any $x_s,x_t\in \Omega_h$, $G(x,x_s)$ solves
$$
\Delta G(x,x_s) + k^2G(x,x_s) = -\delta(x-x_s),\quad x\in   \Omega_h.
$$
Multiplying both sides by $\overline{G(x,x_t)}$ and integrating over $ \Omega_h$ gives
\begin{equation}\label{lemma1_01}
\int_{ \Omega_h} (\Delta G(x,x_s) + k^2G(x,x_s))\overline{G(x,x_t)}\ dx = -\overline{G(x_s,x_t)}.
\end{equation}
Similarly, $\overline{G(x,x_t)}$ solves
$$
\Delta \overline{G(x,x_t)} + k^2\overline{G(x,x_t)} = -\delta(x-x_t),\quad x\in \Omega_h,
$$
thus by multiplying both sides by $G(x,x_s)$ and integrating over $ \Omega_h$ we obtain
\begin{equation}\label{lemma1_02}
\int_{ \Omega_h} (\Delta \overline{G(x,x_t)} + k^2\overline{G(x,x_t)} )G(x,x_s)\ dx = -G(x_t,x_s).
\end{equation}
Subtracting (\ref{lemma1_02}) from (\ref{lemma1_01}) yields
\begin{align*}
G(x_t, x_s)-\overline{G(x_s, x_t)} &= \int_{ \Omega_h}\Delta G(x,x_s)\overline{G(x,x_t)} - \Delta \overline{G(x,x_t)} G(x,x_s)\ dx \\
&= \int_{\partial \Omega_h} \overline{G(x,x_t)} \frac{\partial G(x, x_s)}{\partial n} -G(x, x_s)\frac{\partial \overline{G(x, x_t)} }{\partial n}\ ds(x),
\end{align*}
where $n$ is an unit normal outward vector to $\partial  \Omega_h$. Note that $\partial  \Omega_h = \Gamma \cup \{\pm \pi\}\times (-h,h)$, but thanks to the $\alpha$-quasiperiodicity of $G$
$$
\int_{\{\pm \pi\}\times (-h,h)}\overline{G(x,x_t)} \frac{\partial G(x, x_s) }{\partial n}  -  G(x, x_s)\frac{\partial \overline{G(x, x_t)}}{\partial n}    \ ds(x) = 0.
$$
Therefore
\begin{equation}\label{lemma1_03}
G(x_t, x_s)  -  \overline{G(x_s, x_t)} = \int_{\Gamma_h \cup \Gamma_{-h}} \overline{G(x,x_t)} \frac{\partial G(x, x_s) }{\partial n} -  G(x, x_s)\frac{\partial \overline{G(x, x_t) }}{\partial n}  \ ds(x).
\end{equation}
Define
$$
F(x,y) := \frac{G(x,y) - \overline{G(y, x)}}{2i},\qquad x,y\in\Omega_h,
$$ 
then the left-hand side of \eqref{lemma1_03} equals $2iF(x_t,x_s)$. Now we calculate the right-hand side. Since $G$ satisfies the radiation condition, for a fixed $y\in  \Omega_h$,
\begin{equation}\label{rcg}
G(x, y)=\sum_{j\in \mathbb{Z}}{g}_j^{\pm}(y) e^{i \alpha_j x_1},\ x \in \Gamma_{\pm h}
\end{equation}
and
\begin{equation*}
\frac{\partial G(x, y)}{\partial n} = \sum_{j\in \mathbb{Z}} i \beta_j  {g}_j^{\pm}(y) e^{i \alpha_j x_1},\ x \in \Gamma_{\pm h}.
\end{equation*}
Therefore 
\begin{align*}
\int_{\Gamma_{\pm h}} \overline{G(x,x_t)} \frac{\partial G(x, x_s) }{\partial n}  \ ds(x) &= \int_{-\pi}^{\pi}\sum_{j_1,j_2\in \mathbb{Z}} i\beta_{j_2}\overline{{g}_{j_1}^{\pm}(x_t)}{g}_{j_2}^{\pm}(x_s)  e^{i (\alpha_{j_2}-\alpha_{j_1}) x_1}\ dx_1 \\
&= \sum_{j_1,j_2\in \mathbb{Z}}i\beta_{j_2}\overline{{g}_{j_1}^{\pm}(x_t)} {g}_{j_2}^{\pm}(x_s)  \int_{-\pi}^{\pi} e^{i (j_2-j_1) x_1}\ dx_1,
\end{align*}
and since
$$
\int_{-\pi}^{\pi} e^{i (j_2-j_1) x_1}\ dx_1 = 
\begin{cases} 
2\pi, &j_1 = j_2\\
0, &j_1 \neq j_2
\end{cases}
$$
we have
\begin{equation}\label{lemma1_1}
\int_{\Gamma_h \cup \Gamma_{-h}} \overline{G(x,x_t)} \frac{\partial G(x, x_s) }{\partial n} \ ds(x) = 2\pi i\sum_{j\in \mathbb{Z}}\beta_{j} \left(\overline{{g}_{j}^+(x_t)} {g}_{j}^+(x_s) + \overline{{g}_{j}^-(x_t)} {g}_{j}^-(x_s)\right).
\end{equation}
Similarly
\begin{equation}\label{lemma1_2}
\int_{\Gamma_h \cup \Gamma_{-h}}  G(x, x_s)\frac{\partial \overline{G(x, x_t)}}{\partial n}  \ ds(x) = -2\pi i\sum_{j\in \mathbb{Z}}\overline{\beta_{j}}  \left(\overline{{g}_{j}^+(x_t)} {g}_{j}^+(x_s) + \overline{{g}_{j}^-(x_t)} {g}_{j}^-(x_s)\right).
\end{equation}
Combining (\ref{lemma1_1}) and (\ref{lemma1_2}) yields
$$
G(x_t, x_s) -  \overline{G(x_s, x_t)} = 2\pi i \sum_{j\in \mathbb{Z}}(\beta_{j}+\overline{\beta_{j}})\left(\overline{{g}_{j}^+(x_t)} {g}_{j}^+(x_s) + \overline{{g}_{j}^-(x_t)} {g}_{j}^-(x_s)\right),
$$
that is
\begin{align}
\label{id1}
F(x_t, x_s) = 2\pi \sum_{j\in \mathbb{Z}} \Re \beta_j \left(\overline{{g}_{j}^+(x_t)} {g}_{j}^+(x_s) + \overline{{g}_{j}^-(x_t)} {g}_{j}^-(x_s)\right).
\end{align}
From the Rayleigh expansion~\eqref{eq:radiationCondition} we have that 
$$
{u}_j^{\pm}  = \frac{1}{2\pi}\int_{\Gamma_{\pm h}} u_{sc}(x) e^{-i\alpha_jx_1} ds(x).
$$
Thus, for $j\in\mathbb{Z}$, substituting (\ref{lse}) into the integral above gives
\begin{align*}
{u}_j^{\pm} &= \frac{1}{2\pi}\int_{\Gamma_{\pm h}} \left( k^2\int_{D} G(x,y)q(y)u(y) \ dy\right) e^{-i\alpha_jx_1}\ ds(x)\\
&= k^2\int_{D}\left(\frac{1}{2\pi}\int_{\Gamma_{\pm}}G(x,y)e^{-i\alpha_jx_1}\ ds(x)\right) q(y)u(y)\ dy = k^2 \int_{D} g_j^{\pm}(y)q(y)u(y)\ dy.
\end{align*}
%
%
For $z\in  \Omega$, using the Rayleigh coefficients of ${u}_j^{\pm}$ above and   \eqref{id1}  we obtain
\begin{align*}
&2\pi \sum_{j \in \Z}\Re\beta_j \left({u}^+_j \overline{g^+_j(z)} + {u}^-_j \overline{g^-_j(z)} \right) \\
&= k^2\int_{D}\ 2\pi  \sum_{j \in \Z}\Re\beta_j  \left( g_j^{+}(y)\overline{g^{+}_j(z)} + g_j^-(y)\overline{g^-_j(z)}  \right)q(y)u(y)\ dy \\
&= k^2 \int_{D}  F(z, y)q(y)u(y)\ dy.
\end{align*}
Now, using the following representation of the $\alpha$-quasiperiodic Green function 
$$
G(z, y) = \frac{i}{4} \sum_{j\in \Z}e^{-i2\pi j \alpha}  H^{(1)}_0 \left(k\sqrt{(z_1-y_1 + 2j\pi)^2 + (z_2 - y_2)^2} \right) 
$$
where the series converge for $(y_1- z_1, y_2 - z_2) \neq (2j \pi, 0)$, $j\in \Z$, we have
\begin{align*}
F(z, y) &= \frac{1}{8}\sum_{j\in \Z}e^{-i2\pi j \alpha}  H^{(1)}_0 \left(k\sqrt{(z_1-y_1 + 2j\pi)^2 + (z_2 - y_2)^2} \right) \\
& + \frac{1}{8}\sum_{j\in \Z}e^{i2\pi j \alpha}  \overline{H^{(1)}_0 \left(k\sqrt{(y_1-z_1 + 2j\pi)^2 + (y_2 - z_2)^2} \right)} \\ 
&= \frac{1}{4} \sum_{j\in \Z}e^{-i2\pi j \alpha}  J_0 \left(k\sqrt{(z_1-y_1 + 2j\pi)^2 + (z_2 - y_2)^2} \right). 
\end{align*}

\end{proof}

In the next theorem we will establish a stability estimate for the indicator function.  
Assume $u_{sc}(\cdot,l) \in L^2(\Gamma_r\cup \Gamma_{-r})$ for all $l=1,\dots, N$. 


\begin{theorem}
For $\delta>0$, denote by $u_{{sc},\delta}$ and $\left({u}^\pm_{\delta,j}\right)_{j\in\mathbb{Z}}$ the noisy scattered wave and its Rayleigh sequences respectively, for which we have
$$
\sum_{l=1}^N \| u_{{sc},\delta}(\cdot,l) - u_{sc}(\cdot,l)\|_{L^2(\Gamma_r\cup \Gamma_{-r})} \leq \delta.
$$
Define
$$
I_\delta(z) := \sum_{l=1}^N\left| \sum_{j:\beta_j>0}\beta_j \left({u}^+_{\delta,j}(l) \overline{{g}^+_j(z)} + {u}^-_{\delta,j}(l) \overline{{g}^-_j(z)}\right) \right|^p.
$$
Then the following stability property holds
$$
|I_\delta(z) - I(z)| = O(\delta), \quad \text{as}\ \delta\to 0
$$
for every $z\in  \Omega$.
\end{theorem}

\begin{proof}
For $l=1,\dots,N$ and $j\in\mathbb{Z}$ such that $\beta_j >0$,  we have
$$
|{u}^\pm_{\delta,j}(l) - {u}_j^\pm(l)| \leq \frac{1}{2\pi} \int_{\Gamma_{\pm r}}|u_{{sc}}^\delta(x,l) - u_{sc}(x,l)|\ ds(x),
$$
thus, by Cauchy-Schwarz inequality,
$$
\sum_{l=1}^N |{u}^\pm_{\delta,j}(l) - {u}_j^\pm(l)|  \leq \sum_{l=1}^N \| u_{{sc}}^\delta(\cdot,l) - u_{sc}(\cdot,l)\|_{L^2(\Gamma_r\cup \Gamma_{-r})} \leq \delta.
$$
Therefore
\begin{align*}
&\left|\ \sum_{l=1}^N \left| \sum_{j:\beta_j>0}\beta_j \left({u}^+_{\delta,j}(l) \overline{{g}^+_j(z)} + {u}^-_{\delta,j}(l) \overline{{g}^-_j(z)}\right) \right| - \sum_{l=1}^N \left| \sum_{j:\beta_j>0}\beta_j \left({u}^+_{j}(l) \overline{{g}^+_j(z)} + {u}^-_{j}(l) \overline{{g}^-_j(z)}\right) \right|\ \right| \\
&\leq \sum_{l=1}^N \left|\  \left| \sum_{j:\beta_j>0}\beta_j \left({u}^+_{\delta,j}(l) \overline{{g}^+_j(z)} + {u}^-_{\delta,j}(l) \overline{{g}^-_j(z)}\right) \right| - \left| \sum_{j:\beta_j>0}\beta_j \left({u}^+_{j}(l) \overline{{g}^+_j(z)} + {u}^-_{j}(l) \overline{{g}^-_j(z)}\right) \right|\ \right| \\
& \leq \sum_{l=1}^N\sum_{j:\beta_j>0}\beta_j\left( \left|{u}^+_{\delta,j}(l)- {u}_j^+(l)\right| \left| {g}^{+}_j(z) \right| + \left|{u}^-_{\delta,j}(l)- {u}_j^-(l) \right| \left| {g}^{-}_j(z) \right| \right)\\
&\leq C\delta,
\end{align*}
where
$$
C = \left(\max_{j:\beta_j>0}\left\| {g}^{+}_j \right\|_{L^\infty(\Omega)} + \max_{j:\beta_j>0} \left\| {g}^{-}_j \right\|_{L^\infty(\Omega)}\right) \sum_{j:\beta_j>0}\beta_j.
$$
For brevity, set
\begin{align*}
a_l&:= \left| \sum_{j:\beta_j>0}\beta_j \left({u}^+_{\delta,j}(l) \overline{{g}^+_j(z)} + {u}^-_{\delta,j}(l) \overline{{g}^-_j(z)}\right) \right|, \\
b_l&:= \left| \sum_{j:\beta_j>0}\beta_j \left({u}^+_{j}(l) \overline{{g}^+_j(z)} + {u}^-_{j}(l) \overline{{g}^-_j(z)}\right) \right|, \\
\gamma &:= \max_{\substack{ j:\beta_j>0\\ l=1,\dots,N}}\left\{ \left| {u}^{+}_j(l) \right|, \left| {u}^{-}_j(l) \right| \right\},
\end{align*}
then we have 
$$
\sum_{l=1}^N\left|a_l-b_l\right| \leq C\delta, \quad b_l \leq C\gamma,
$$
 for all $l=1,\dots,N$. Hence, for all $z \in \Omega$,
\begin{align*}
|I_\delta(z) - I(z)| &\leq \sum_{l=1}^N \left|a_l^p - b_l^p\right| = \sum_{l=1}^N \left(\left|a_l-b_l\right|\left| \sum_{m=0}^{p-1} a_l^m b_l^{p-1-m}\right|\right) \\ 
&\leq \left(\sum_{l=1}^N\left|a_l-b_l\right| \right) \left(\sum_{l=1}^N \left| \sum_{m=0}^{p-1} a_l^m b_l^{p-1-m}\right|\right) \\
&\leq C\delta \sum_{l=1}^N \sum_{m=0}^{p-1} (\left|a_l-b_l\right|+b_l)^m b_l^{p-1-m} \\
&\leq C\delta \sum_{l=1}^N \sum_{m=0}^{p-1}2^m (\left|a_l-b_l\right|^m+b_l^m)b_l^{p-1-m} \\
&\leq C\delta \sum_{m=0}^{p-1} \sum_{l=1}^N 2^m (\left|a_l-b_l\right|^m+\gamma^m)\gamma^{p-1-m} \\
&\leq C\delta \sum_{m=0}^{p-1}2^m \left(C^m\delta^m \gamma^{p-1-m}+N\gamma^{p-1}\right)\\
& = O(\delta),\quad \text{as}\ \delta\to 0.
\end{align*}
This completes the proof.
\end{proof}

\section{Numerical study}
\label{numerical}

In this section, we test the performance of the proposed sampling method with respect to  
the number of incident sources, the levels of noise in the data, the wave numbers, the values of parameter $\alpha$ (in~\eqref{eq:radiationCondition})  and exponent $p$, and the shape
of the periodic scatterers. For the latter category we compare the performance of the  proposed sampling method  with those of the factorization method and the orthogonality sampling method.

In all of the numerical examples below we choose the following parameters:
\begin{align*}
\text{sampling domain} = (-\pi,\pi)\times (-1,1),  \\
  \text{ measurement  boundary }\ \Gamma_{\pm 2} = (-\pi,\pi)\times \{\pm 2\}, \\
  \text{location of incident sources } \ \Gamma_{\pm 3} = (-\pi,\pi)\times \{\pm 3\}.
\end{align*}
The sampling domain is uniformly discretized in each dimension  with $128 \times 96$ sampling points. 
The boundary measurements  $\Gamma_{\pm 2} $ are discretized uniformly with 64 points on each  boundary. 
 By using $N$ incident sources we mean to consider
 $$
u_{in}(x,l) = G(x,s_l), \quad x \in \Omega,\quad s_l\in \Gamma_{\pm 3},\ l=1,\dots,N,
$$
where  $N/2$ sources are uniformly located on $\Gamma_{+ 3} $  and $N/2$ sources are uniformly located on $\Gamma_{- 3} $. 
We generate the synthetic scattering data by solving the direct problem with the spectral Galerkin method studied in~\cite{Lechl2014}.
With  artificial noise added to the synthetic scattering data we implement the indicator function 
$$
I_\delta(z) = \sum_{l=1}^{N}\left| \sum_{j:\beta_j >0 } \beta_j \left({u}^+_{\delta,j}(l) \overline{{g}^+_j(z)} + {u}^-_{\delta,j}(l) \overline{{g}^-_j(z)}\right) \right|^p.
$$

To compare with the orthogonality sampling method we implement following indicator function
$$
I_{OSM}(z) = \sum_{l=1}^N \left| \int_{\Gamma_{+ 2} \cup \Gamma_{- 2}}u_{sc,\delta}(x) \, \overline{G(x,z)}ds(x)  \right|^p.
$$ 
This indicator function of the orthogonality sampling method can be rewritten in the modal form by using the Rayleigh expansion of the scattered field. 
In this modal form the evanescent modes or the exponentially decaying terms can be neglected and there remains only a finite number of  the propagating modes. If we drop $\beta_j$ in $I_\delta(z)$, we will 
 approximately obtain  the modal form of $I_{OSM}(z)$.  We refer to~\cite{Nguye2014} for the indicator function of 
the factorization method implemented for the test in this section.

\subsection{Reconstruction with one incident source (Figure~\ref{fi1})} 
In this section we test the performance of the sampling method for data generated by only one incident source. 
Here the parameters are chosen as $\alpha = 0$,   wave number $k = 2\pi$, and  exponent $p = 4$. We  add 20$\%$ artificial noise to the scattered field data  ($\delta = 20\%$). 
From Figure~\ref{fi1} we can see that the method is able to reconstruct small scatterers quite well.
This capability is an advantage over classical sampling methods (e.g. linear sampling method, factorization method) in terms of computational efficiency since classical methods are only able to reconstruct targets with data generated by multiple incident fields. 
However, the proposed sampling method fails to reconstruct  scatterers with extended shape, which is reasonable since only one incident source and one wave number are used for the data. 
We refer to Figure~\ref{fig2} for improved results when multiple incident sources are used for the reconstruction.

\begin{figure}[ht!]
\centering

\begin{subfigure}{\textwidth}
\centering
\includegraphics[trim=0 250 0 250,clip,width= 0.85\textwidth]{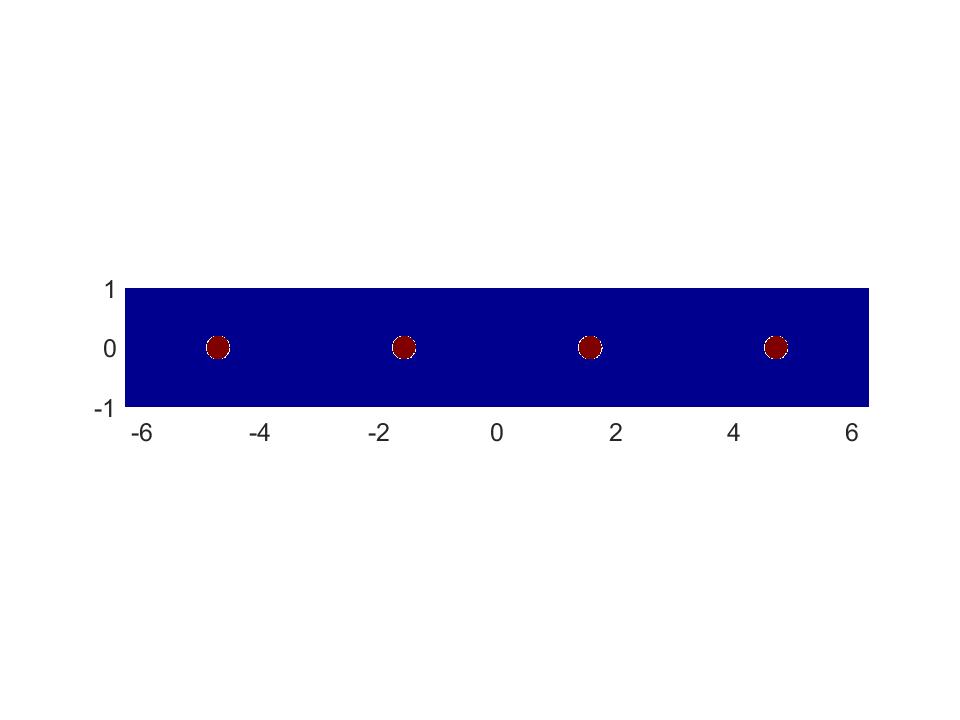}
\caption{True geometry in $(-2\pi,2\pi)$.}
\end{subfigure}

\begin{subfigure}{\textwidth}
\centering
\includegraphics[trim=0 250 0 250,clip,width=0.85\textwidth]{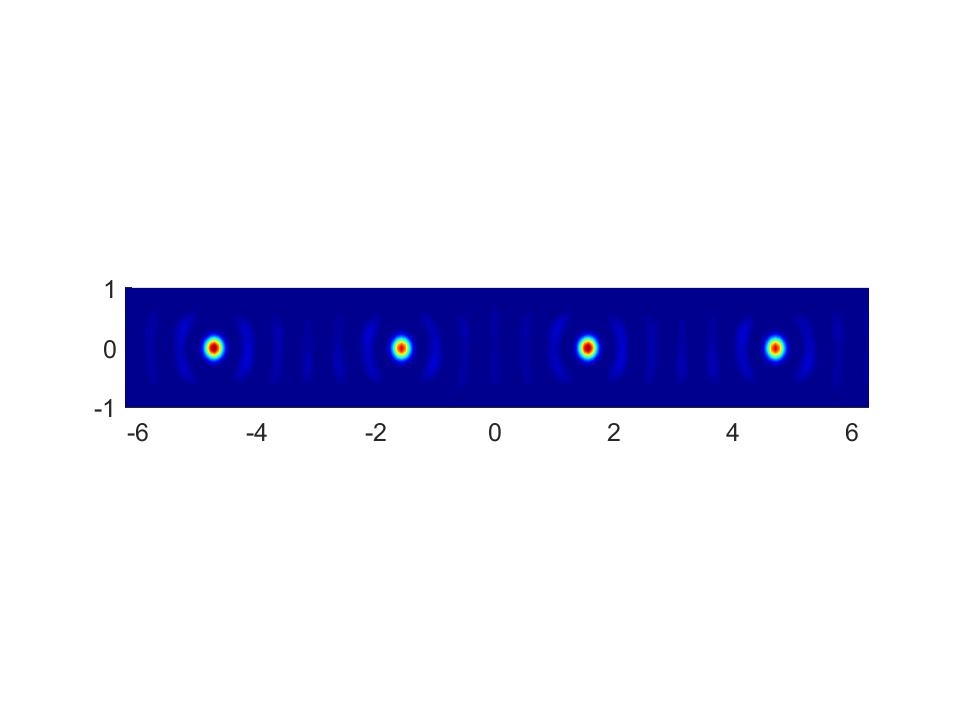}
\caption{Reconstruction of small elliptical scatterers in Figure~\ref{fi1}-(a).}
\end{subfigure}

\begin{subfigure}{\textwidth}
\centering
\includegraphics[trim=0 250 0 250,clip,width=0.85\textwidth]{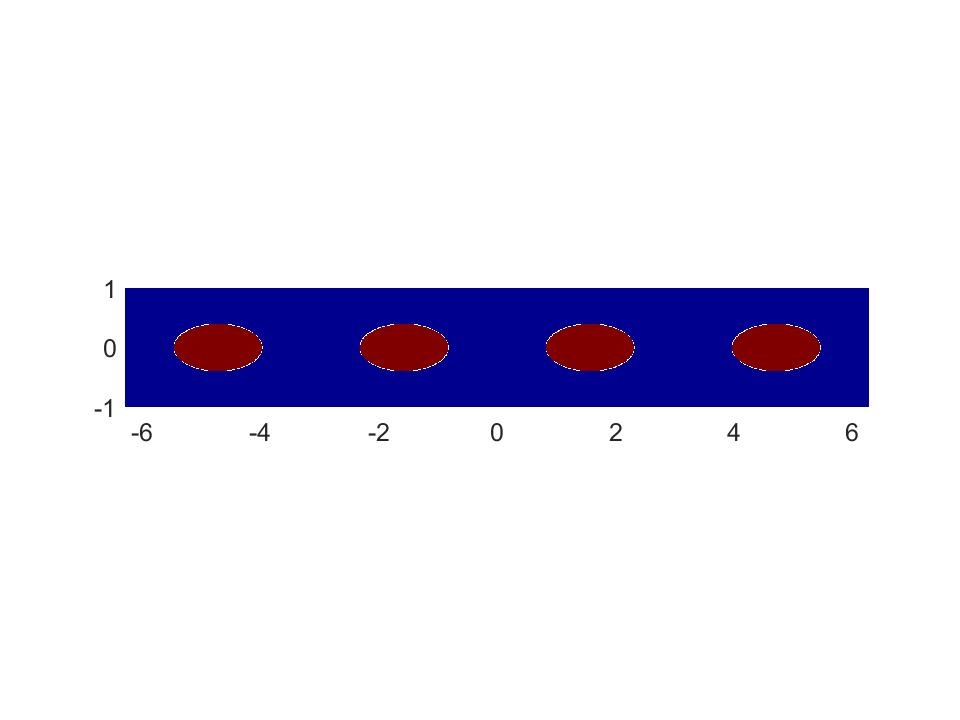}
\caption{True geometry in $(-2\pi,2\pi)$.}
\end{subfigure}

\begin{subfigure}{\textwidth}
\centering
\includegraphics[trim=0 250 0 250,clip,width=0.85\textwidth]{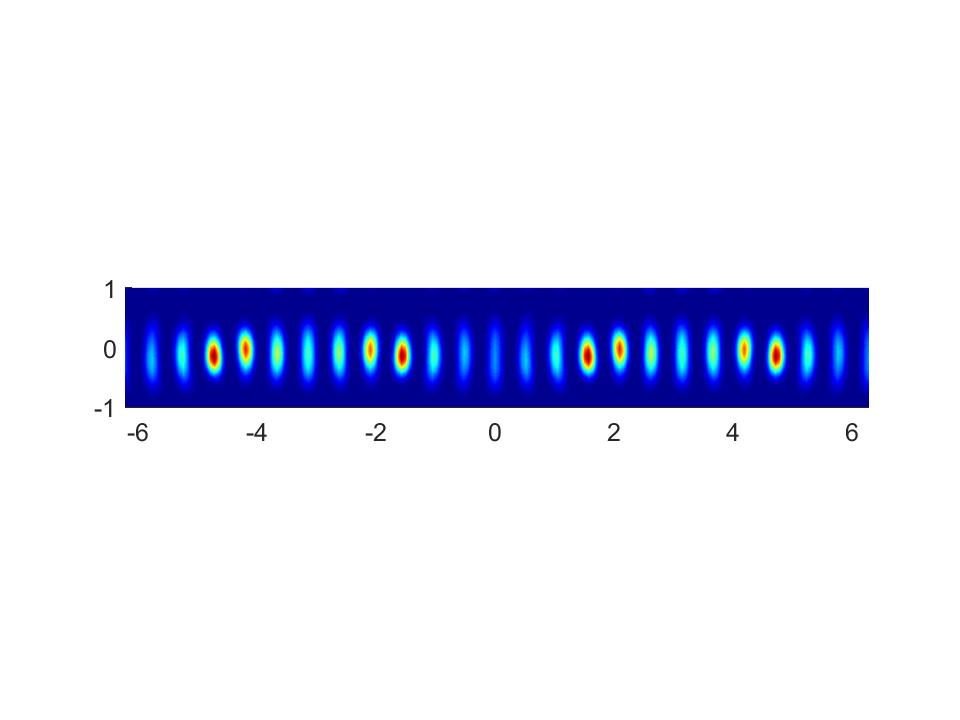}
\caption{Reconstruction of extended elliptical scatterers in Figure~\ref{fi1}-(c).}
\end{subfigure}

\caption{Reconstruction of small and extended  elliptical scatterers with one incident source.}
\label{fi1}

\end{figure}

\subsection{Reconstruction with multiple incident sources (Figure~\ref{fig2})}
In this section we test the performance of the method for reconstructing extended scatterers with different number of incident sources.  
The parameters are the same  as in Figure~\ref{fi1}, meaning $\alpha = 0$,   wave number $k = 2\pi$,  exponent $p = 4$, and 20$\%$ artificial noise is added to the scattered field data.
If in Figure~\ref{fi1} the method fails to reconstruct extended ellipses with one incident source, the reconstructions are improved with more sources, see Figure~\ref{fig2}. 
The reconstructions already look reasonable with 32 incident sources and continue to improve when  more sources are used. The results remain almost the same even if more than 128 incident sources are used
to generate the scattering data. 

\begin{figure}[h!]
\centering

\begin{subfigure}{\textwidth}
\centering
\includegraphics[trim=0 250 0 250,clip,width=0.85\textwidth]{smth_ellips_true}
\caption{True geometry in $(-2\pi,2\pi)$.}
\end{subfigure}

\begin{subfigure}{\textwidth}
\centering
\includegraphics[trim=0 250 0 250,clip,width=0.85\textwidth]{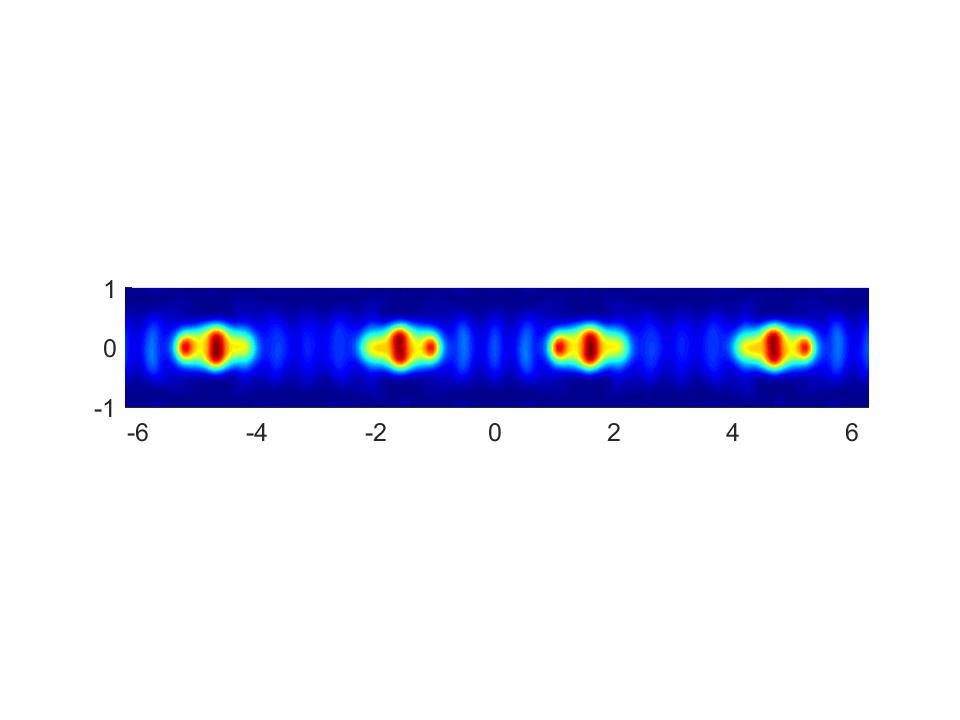}
\caption{Reconstruction with $32$ incident sources.}
\end{subfigure}

\begin{subfigure}{\textwidth}
\centering
\includegraphics[trim=0 250 0 250,clip,width=0.85\textwidth]{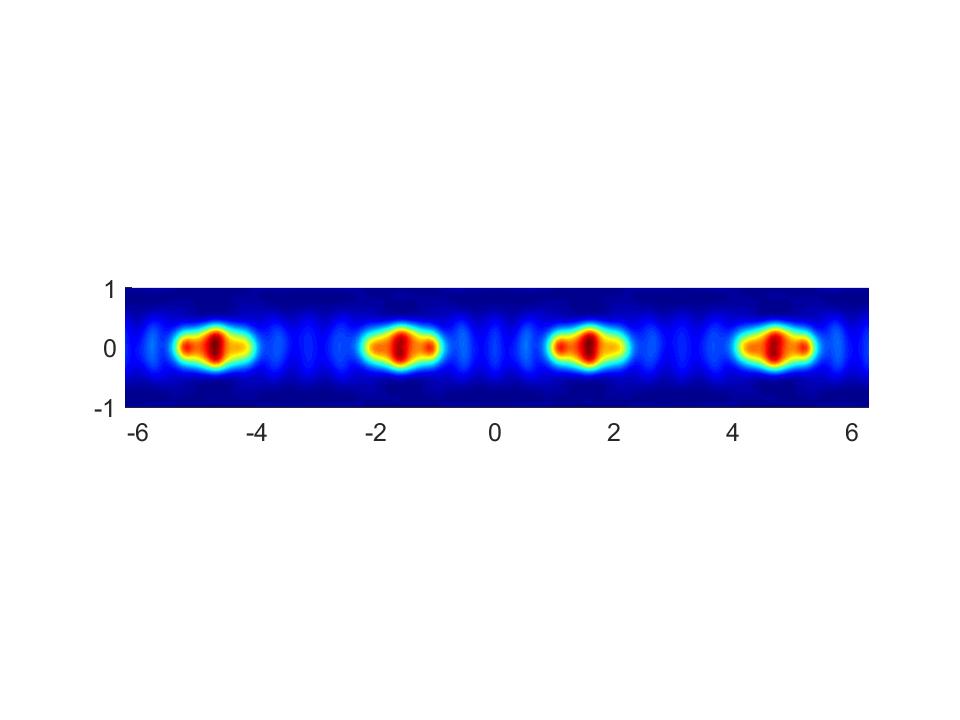}
\caption{Reconstruction with $64$ incident sources.}
\end{subfigure}

\begin{subfigure}{\textwidth}
\centering
\includegraphics[trim=0 250 0 250,clip,width=0.85\textwidth]{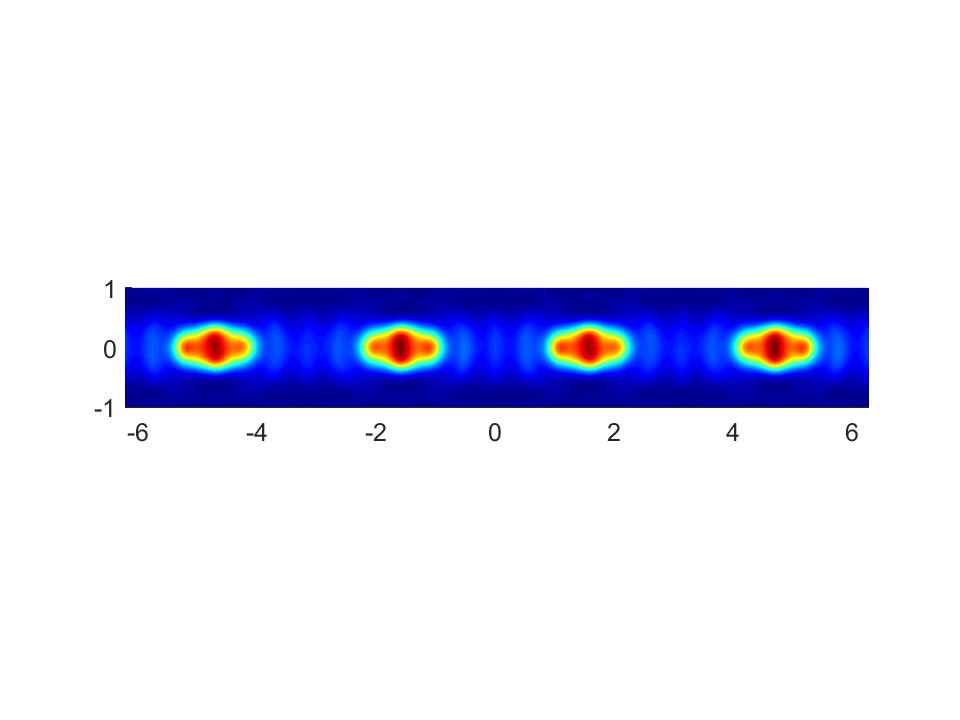}
\caption{Reconstruction with $128$ incident sources.}
\end{subfigure}

\caption{Reconstruction  with different numbers of incident sources.}
\label{fig2}
\end{figure}

\subsection{Reconstruction  with different levels of noise in the data (Figure~\ref{fig3})}
In this section we test the performance of the method for reconstructing extended scatterers with different levels of artificial noise in the data ($10\%, 20\%$ and $40\%$).  
Here the parameters are chosen as $\alpha = 0$,  wave number $k = 2\pi$,  exponent $p = 4$, and 128 incident sources are used to generate the scattering data.
It can be seen from Figure~\ref{fig3} that all reconstructions are not  affected by the amounts of noise added to the data. Furthermore, the reconstructions 
will remain essentially the same even with much higher amounts of noise in the data. This is not a surprise since
 the evanescent modes are typically  sensitive with noise  but the sampling method only uses propagating modes. The great robustness against noise
 in the data was also seen in the orthogonality sampling methods studied in~\cite{Potth2010, Harri2020, Le2022}. 

\begin{figure}[ht!]
\centering

\begin{subfigure}{\textwidth}
\centering
\includegraphics[trim=0 250 0 250,clip,width=0.85\textwidth]{smth_ellips_true}
\caption{True geometry in $(-2\pi,2\pi)$.}
\end{subfigure}

\begin{subfigure}{\textwidth}
\centering
\includegraphics[trim=0 250 0 250,clip,width=0.85\textwidth]{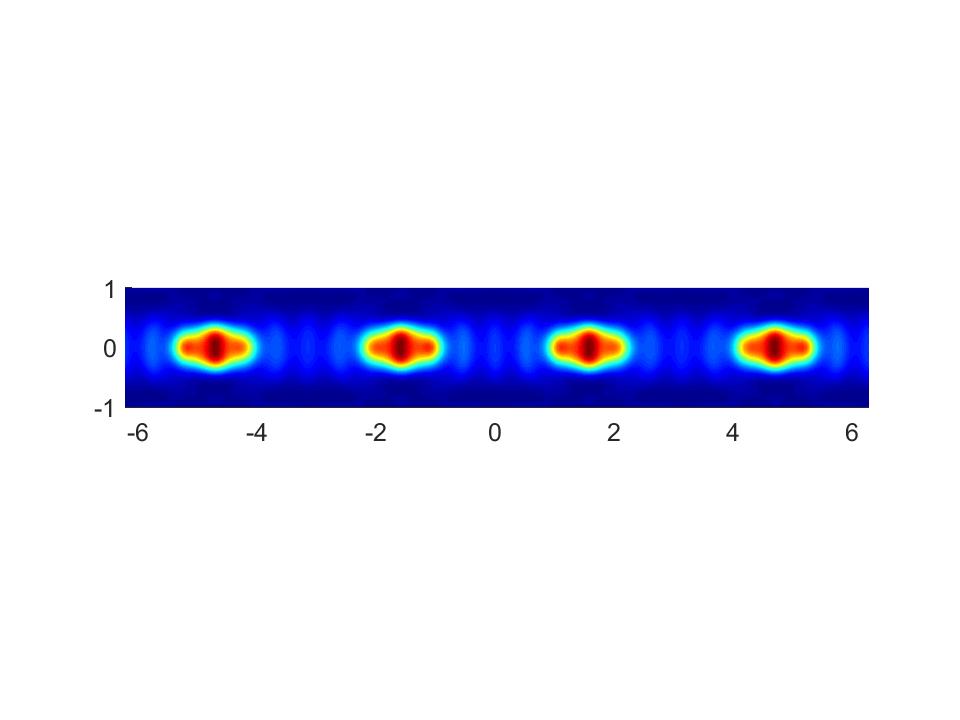}
\caption{Reconstruction with $10\%$ noise.}
\end{subfigure}

\begin{subfigure}{\textwidth}
\centering
\includegraphics[trim=0 250 0 250,clip,width=0.85\textwidth]{smth_ellips_rec}
\caption{Reconstruction with $20\%$ noise.}
\end{subfigure}

\begin{subfigure}{\textwidth}
\centering
\includegraphics[trim=0 250 0 250,clip,width=0.85\textwidth]{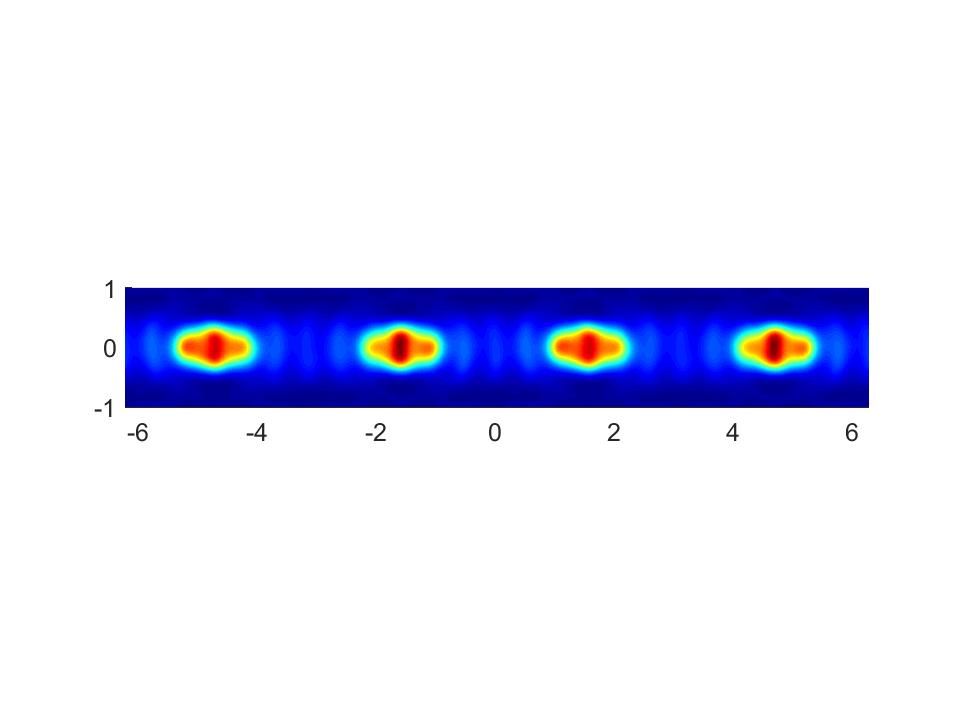}
\caption{Reconstruction with $40\%$ noise.}
\end{subfigure}

\caption{Reconstruction with  different levels of noise in the data.}
\label{fig3}

\end{figure}

\subsection{Reconstruction with different wave numbers (Figure~\ref{fig4})}
In this section we test the performance of the method for reconstructing extended scatterers with different wave numbers ($k = \pi, \,2\pi,\, 3\pi$).  
Here the parameters are chosen as $\alpha = 0$,   exponent $p = 4$,   128 incident sources are used to generate the scattering data, and $20\%$ noise added to the data. 
We can see from~Figure~\ref{fig4} that the resolution of reconstructions improve  as the wave number increases, which is reasonable.  However, it is also known that
 as the wave number increases, the inverse problem will become more difficult to deal with. We can notice some small effect of larger wave numbers in the reconstruction for $k = 3\pi$.

\begin{figure}[h!]
\centering

\begin{subfigure}{\textwidth}
\centering
\includegraphics[trim=0 250 0 250,clip,width=0.85\textwidth]{smth_ellips_true}
\caption{True geometry in $(-2\pi,2\pi)$.}
\end{subfigure}

\begin{subfigure}{\textwidth}
\centering
\includegraphics[trim=0 250 0 250,clip,width=0.85\textwidth]{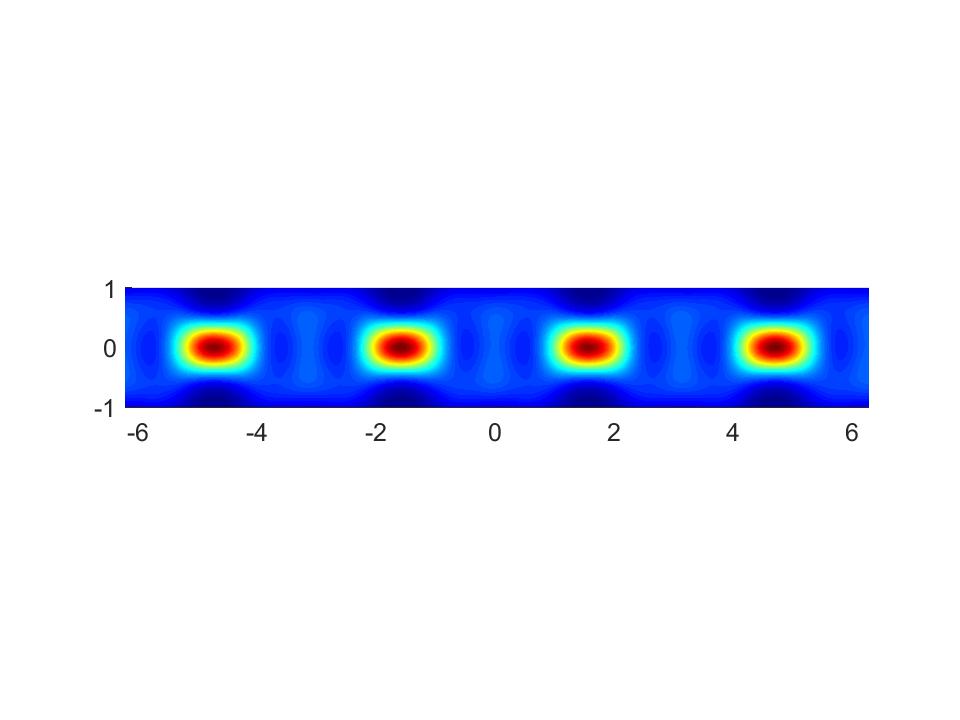}
\caption{Reconstruction with wave number $k = \pi$.}
\end{subfigure}

\begin{subfigure}{\textwidth}
\centering
\includegraphics[trim=0 250 0 250,clip,width=0.85\textwidth]{smth_ellips_rec}
\caption{Reconstruction with  wave number $k = 2\pi$.}
\end{subfigure}

\begin{subfigure}{\textwidth}
\centering
\includegraphics[trim=0 250 0 250,clip,width=0.85\textwidth]{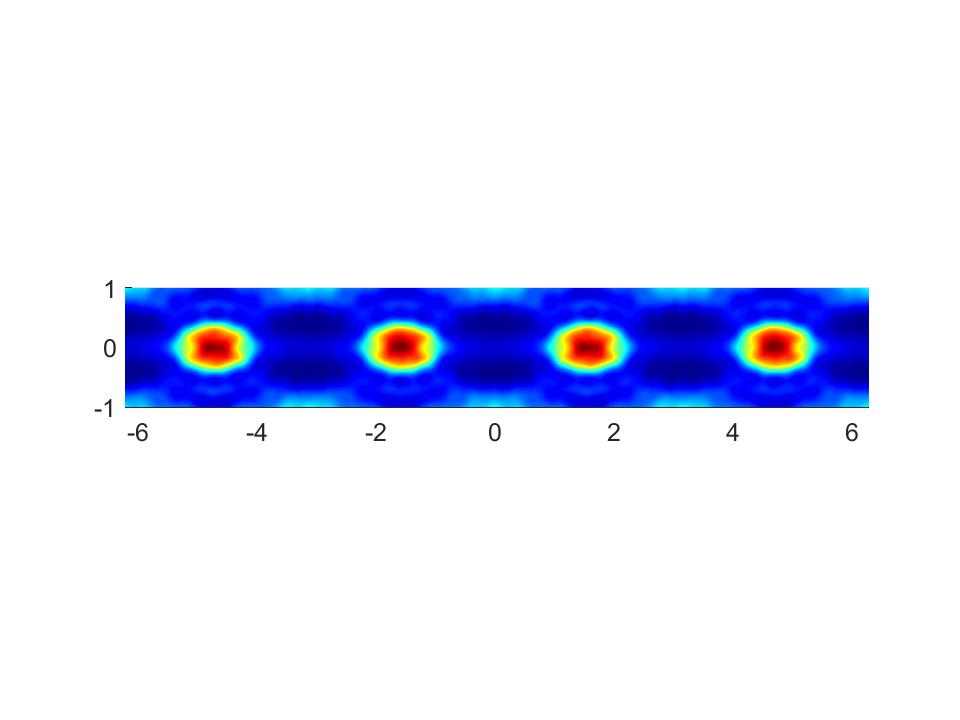}
\caption{Reconstruction with  wave number $k = 3\pi$.}
\end{subfigure}

\caption{Reconstruction with different wave numbers.}
\label{fig4}
\end{figure}

\subsection{Reconstruction with different values of $\alpha$ and  $p$  (Figures~\ref{fig5}--\ref{fig6})}

In this section we test the performance of the method for different values of the parameters $\alpha$ and  $p$. 
The other parameters are chosen as $k = 2\pi$,  128 incident sources are used to generate the scattering data, and $20\%$ 
noise added to the data. From Figures~\ref{fig5}--\ref{fig6}, we can see that the sampling method works well for different 
values of $\alpha$. Although the reconstructions look reasonable for exponents $p = 2$ and $p = 3$, the exponent $p = 4$ 
seems to be an ideal exponent for the indicator function. With larger $p$  the reconstructions will  have cleaner background but lose some small details of the scatterers.

\begin{figure}[h!]
\centering

\begin{subfigure}{\textwidth}
\centering
\includegraphics[trim=0 250 0 250,clip,width=0.85\textwidth]{smth_ellips_true}
\caption{True geometry in $(-2\pi,2\pi)$.}
\end{subfigure}

\begin{subfigure}{\textwidth}
\centering
\includegraphics[trim=0 250 0 250,clip,width=0.85\textwidth]{smth_ellips_rec}
\caption{Reconstruction with parameter $\alpha = 0$.}
\end{subfigure}

\begin{subfigure}{\textwidth}
\centering
\includegraphics[trim=0 250 0 250,clip,width=0.85\textwidth]{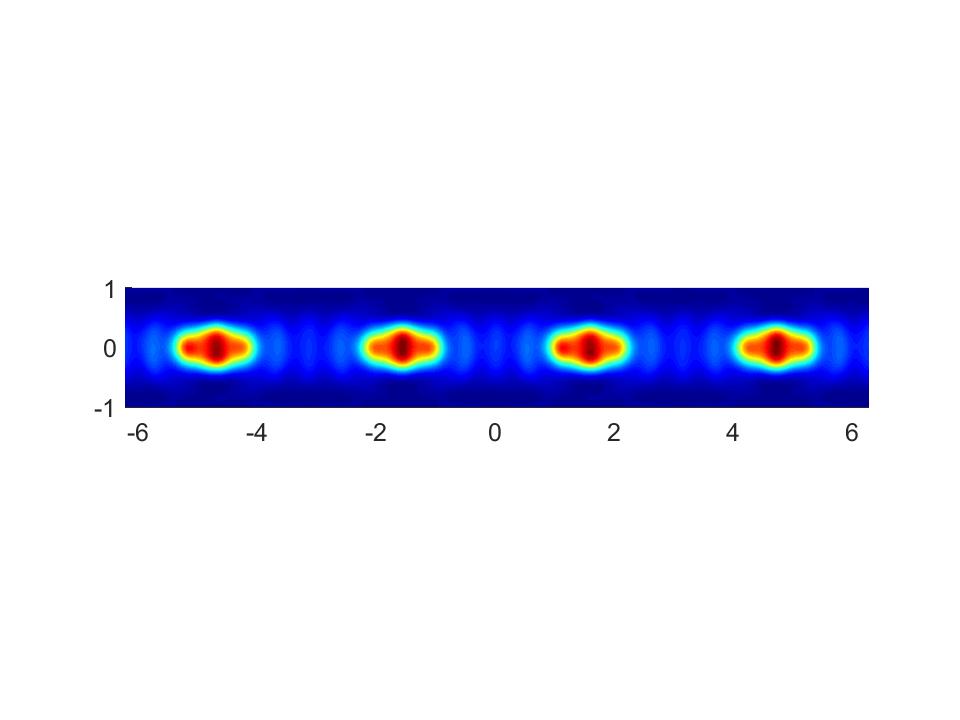}
\caption{Reconstruction with parameter $\alpha = \pi/3$.}
\end{subfigure}

\begin{subfigure}{\textwidth}
\centering
\includegraphics[trim=0 250 0 250,clip,width=0.85\textwidth]{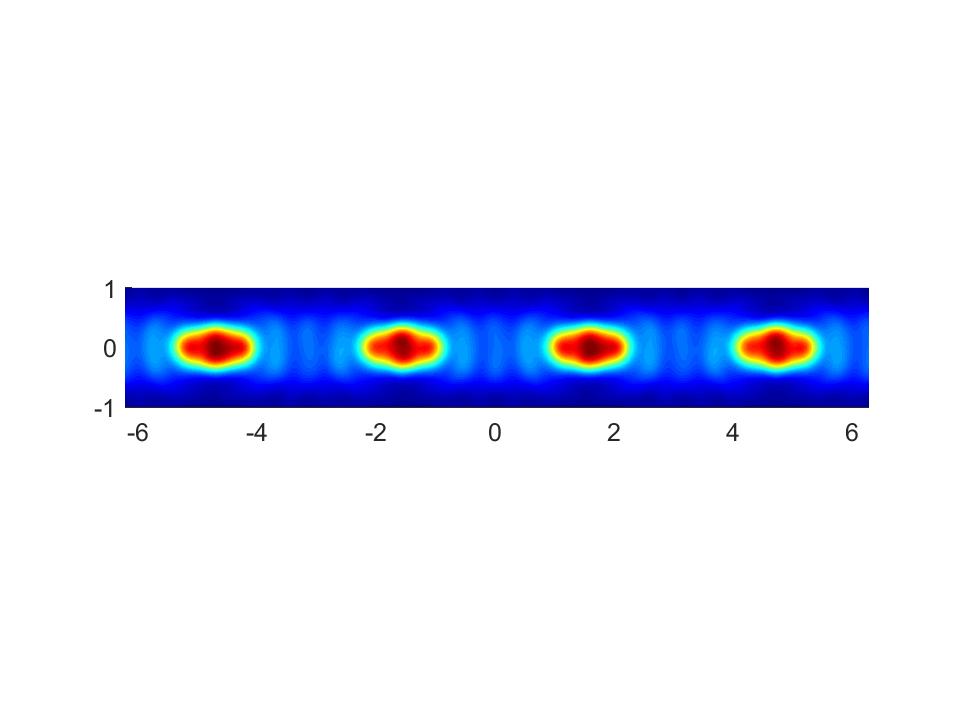}
\caption{Reconstruction with  parameter $\alpha = \pi$.}
\end{subfigure}

\caption{Reconstruction with different values of parameter $\alpha$.}
\label{fig5}
\end{figure}

\begin{figure}[h!]
\centering

\begin{subfigure}{\textwidth}
\centering
\includegraphics[trim=0 250 0 250,clip,width=0.85\textwidth]{smth_ellips_true}
\caption{True geometry in $(-2\pi,2\pi)$.}
\end{subfigure}

\begin{subfigure}{\textwidth}
\centering
\includegraphics[trim=0 250 0 250,clip,width=0.85\textwidth]{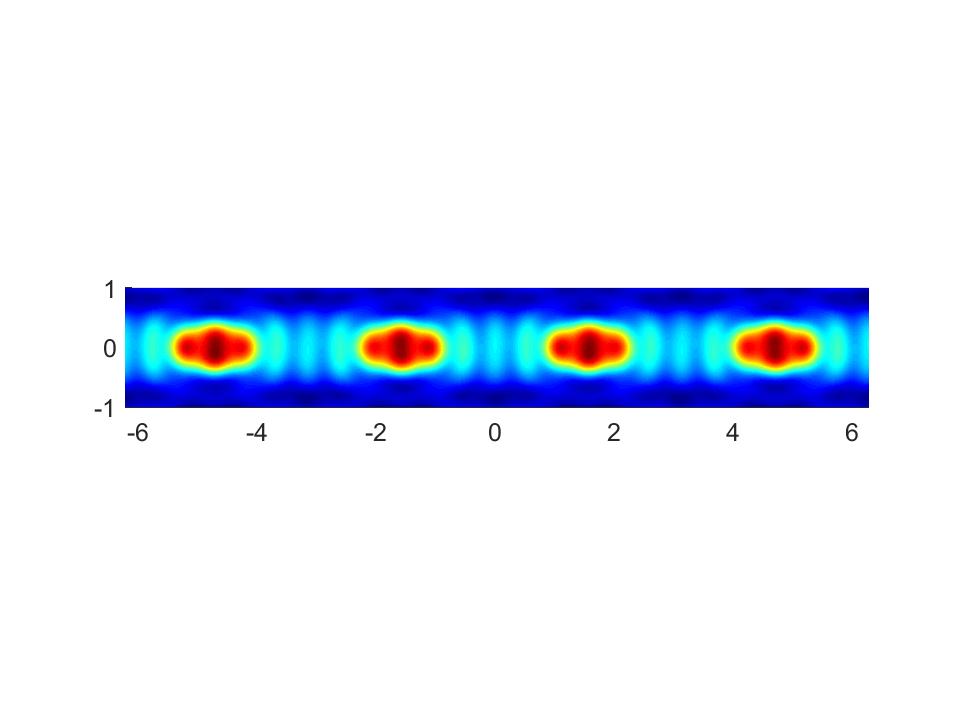}
\caption{Reconstruction with exponent $p = 2$.}
\end{subfigure}

\begin{subfigure}{\textwidth}
\centering
\includegraphics[trim=0 250 0 250,clip,width=0.85\textwidth]{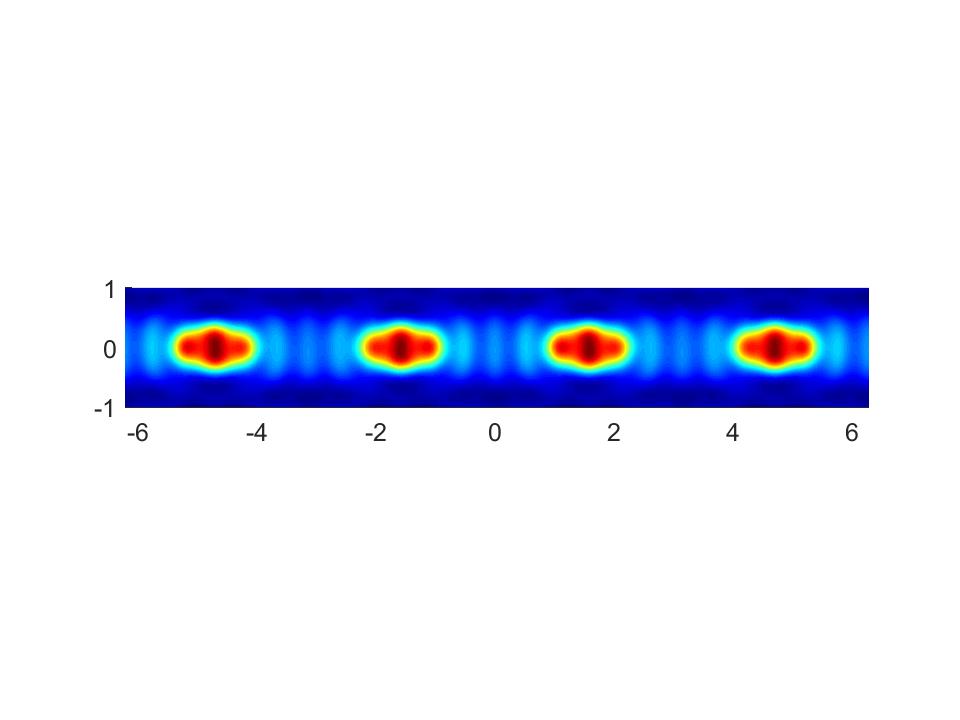}
\caption{Reconstruction with  exponent $p = 3$.}
\end{subfigure}

\begin{subfigure}{\textwidth}
\centering
\includegraphics[trim=0 250 0 250,clip,width=0.85\textwidth]{smth_ellips_rec}
\caption{Reconstruction with exponent $p = 4$.}
\end{subfigure}

\caption{Reconstruction with different exponents  $p$ in the indicator function.}
\label{fig6}
\end{figure}

\subsection{Reconstruction of different shapes and comparison with  the factorization method and the orthogonality sampling method (Figures~\ref{fig7}--\ref{fig10})}
In this section we test the performance of the proposed sampling method for different shapes of periodic scatterers and compare its performance with those of the factorization method and the orthogonality sampling method. 
The parameters are chosen as $\alpha = 0$,  wave number $k = 2\pi$, exponent $p = 4$,   128 incident sources are used to generate the scattering data, and $20\%$ noise added to the data. 
It is obvious from all reconstructions in Figures~\ref{fig7}--\ref{fig10} that the factorization method suffers severely from the $20\%$ amount of noise added to the scattering data. Actually, the reconstructions of the factorization method
can be greatly affected even with smaller amounts of noises (e.g. $5\%$ or $7\%$). This unstable behavior of
the factorization method in reconstructing periodic media was also reported in~\cite{Arens2005, Nguye2016}. The reconstructions of the orthogonality sampling method are as stable as those of the proposed sampling method but it is also clear from the pictures that the accuracy in the reconstructions of the orthogonality sampling method is much worse than that of the proposed sampling method. The proposed sampling method may provide reasonable reconstructions  
for different shapes considered in the test. This indicates a high efficiency of this sampling method.

\begin{figure}[h!]
\centering

\begin{subfigure}{\textwidth}
\centering
\includegraphics[trim=0 250 0 250,clip,width=0.85\textwidth]{smth_ellips_true}
\caption{True geometry in $(-2\pi,2\pi)$.}
\end{subfigure}

\begin{subfigure}{\textwidth}
\centering
\includegraphics[trim=0 250 0 250,clip,width=0.85\textwidth]{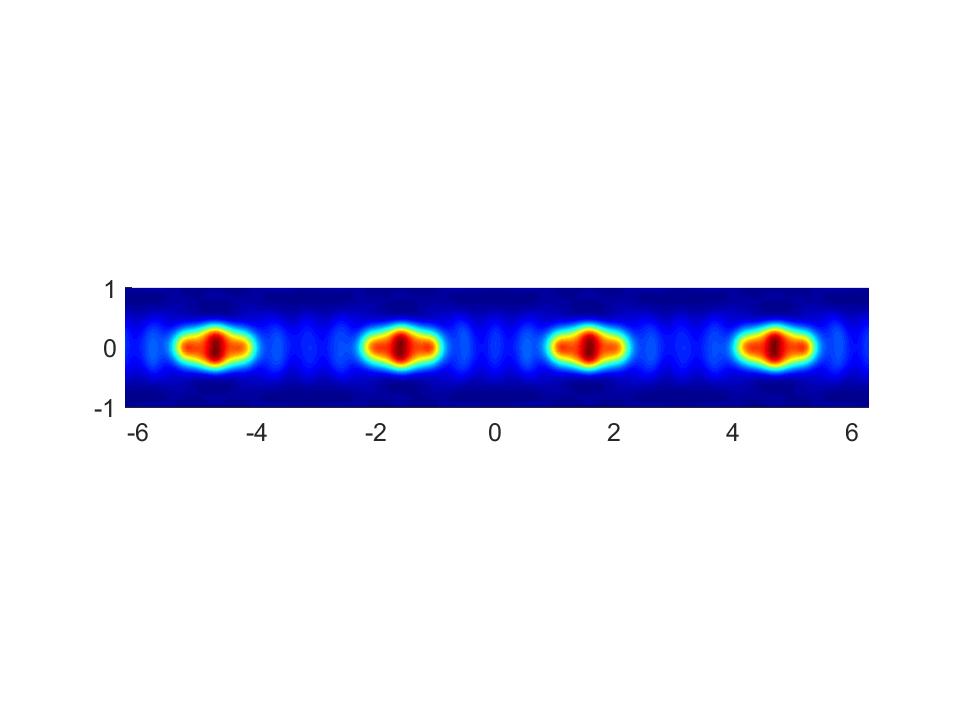}
\caption{Reconstruction by the proposed sampling method.}
\end{subfigure}

\begin{subfigure}{\textwidth}
\centering
\includegraphics[trim=0 250 0 250,clip,width=0.85\textwidth]{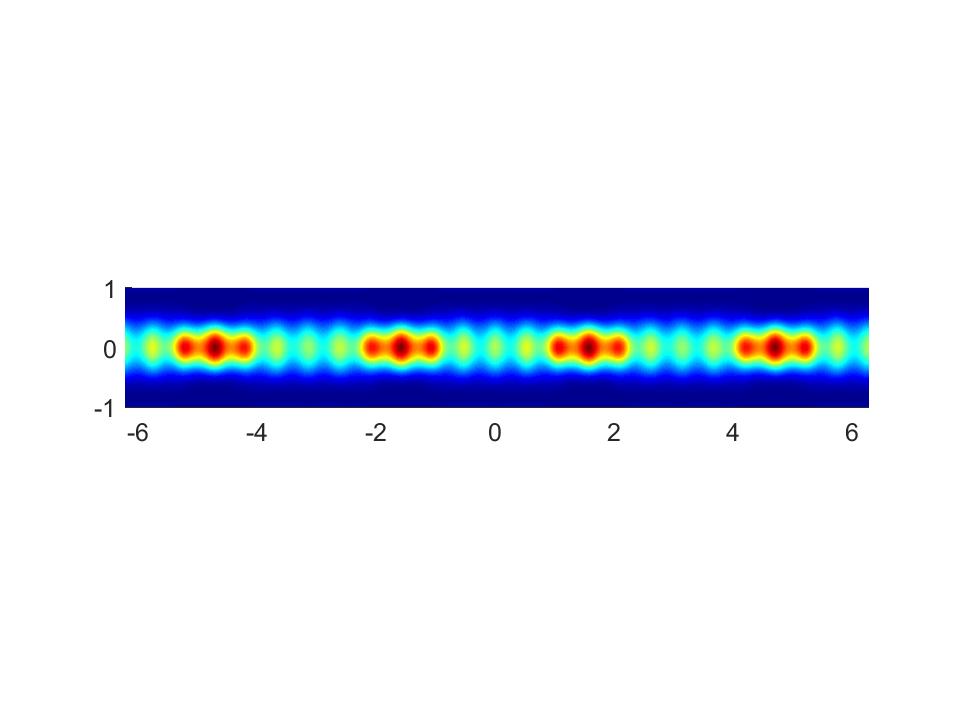}
\caption{Reconstruction by the orthogonality sampling method.}
\end{subfigure}

\begin{subfigure}{\textwidth}
\centering
\includegraphics[trim=0 250 0 250,clip,width=0.85\textwidth]{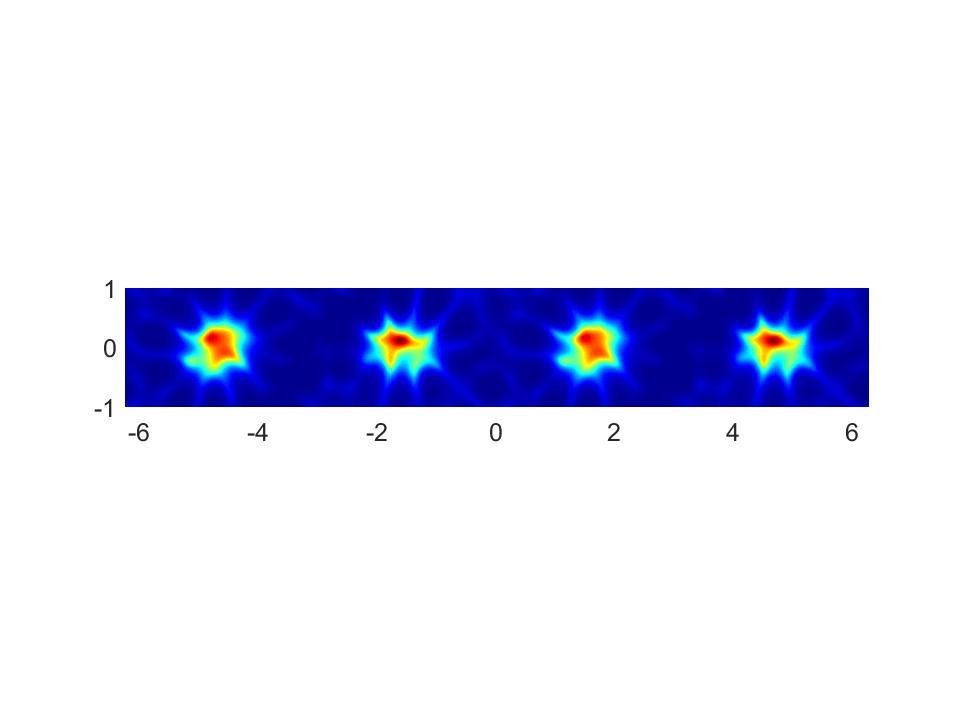}
\caption{Reconstruction by the factorization method.}
\end{subfigure}

\caption{Reconstruction of elliptical scatterers  using different sampling methods.}
\label{fig7}
\end{figure}


\begin{figure}[h!]
\centering

\begin{subfigure}{\textwidth}
\centering
\includegraphics[trim=0 250 0 250,clip,width=0.85\textwidth]{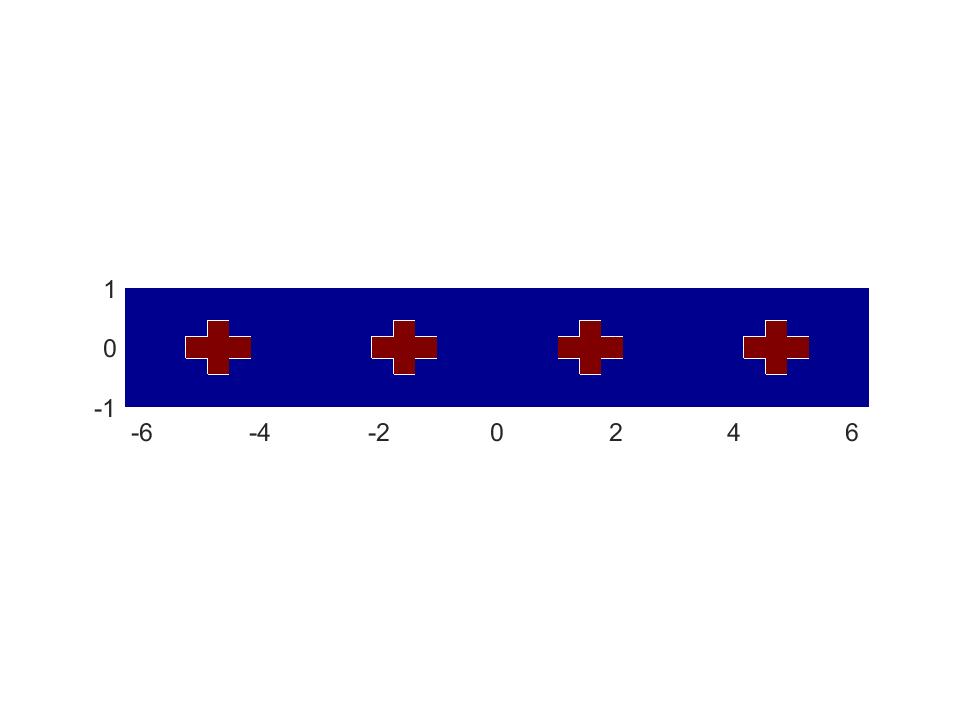}
\caption{True geometry in $(-2\pi,2\pi)$.}
\end{subfigure}

\begin{subfigure}{\textwidth}
\centering
\includegraphics[trim=0 250 0 250,clip,width=0.85\textwidth]{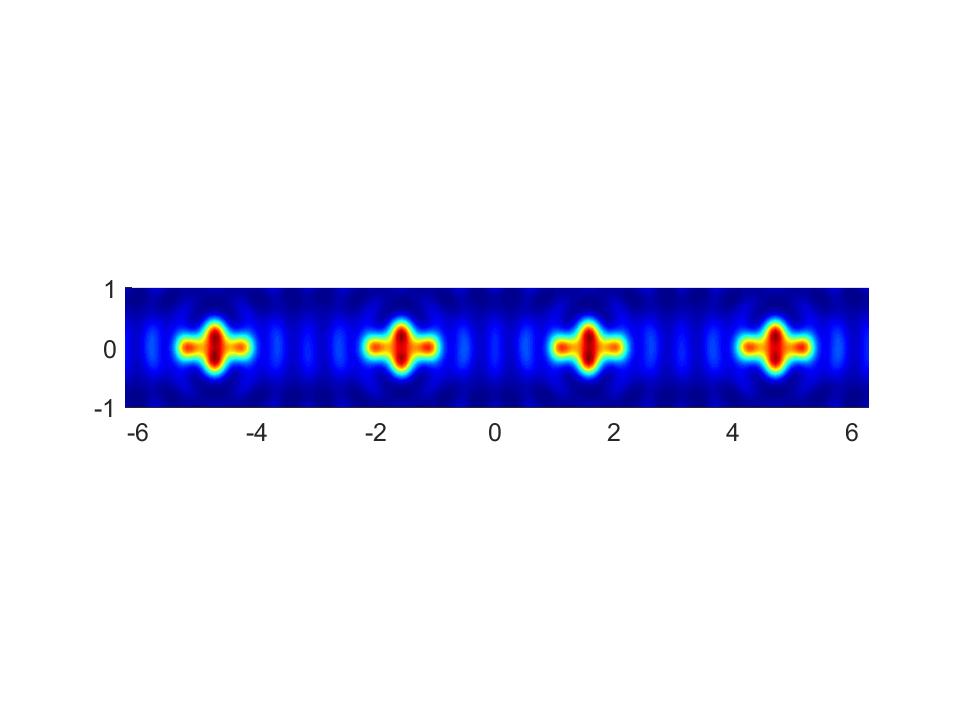}
\caption{Reconstruction by the proposed sampling method.}
\end{subfigure}

\begin{subfigure}{\textwidth}
\centering
\includegraphics[trim=0 250 0 250,clip,width=0.85\textwidth]{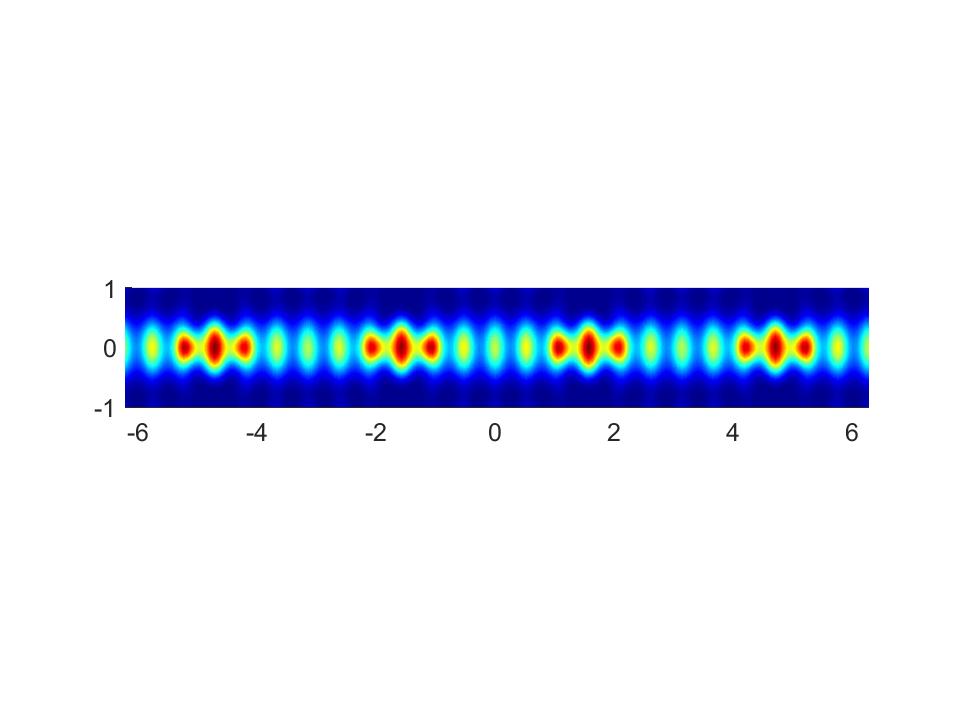}
\caption{Reconstruction by the orthogonality sampling method.}
\end{subfigure}

\begin{subfigure}{\textwidth}
\centering
\includegraphics[trim=0 250 0 250,clip,width=0.85\textwidth]{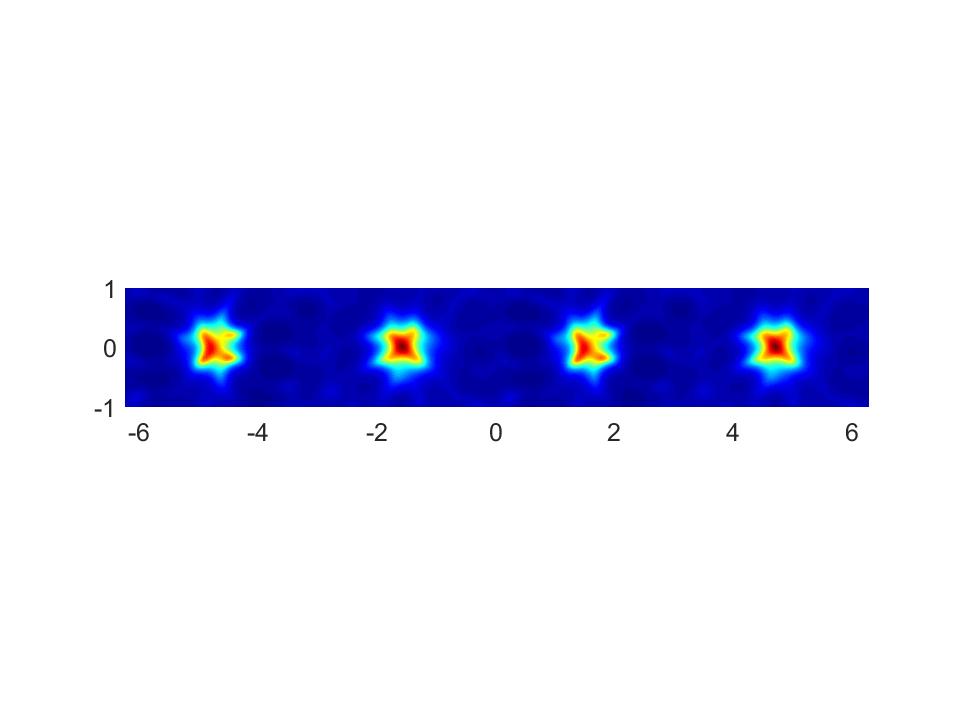}
\caption{Reconstruction by  the factorization method.}
\label{fig8}
\end{subfigure}

\caption{Reconstruction of cross-shaped scatterers using different sampling methods.}
\end{figure}


\begin{figure}[h!]
\centering

\begin{subfigure}{\textwidth}
\centering
\includegraphics[trim=0 250 0 250,clip,width=0.85\textwidth]{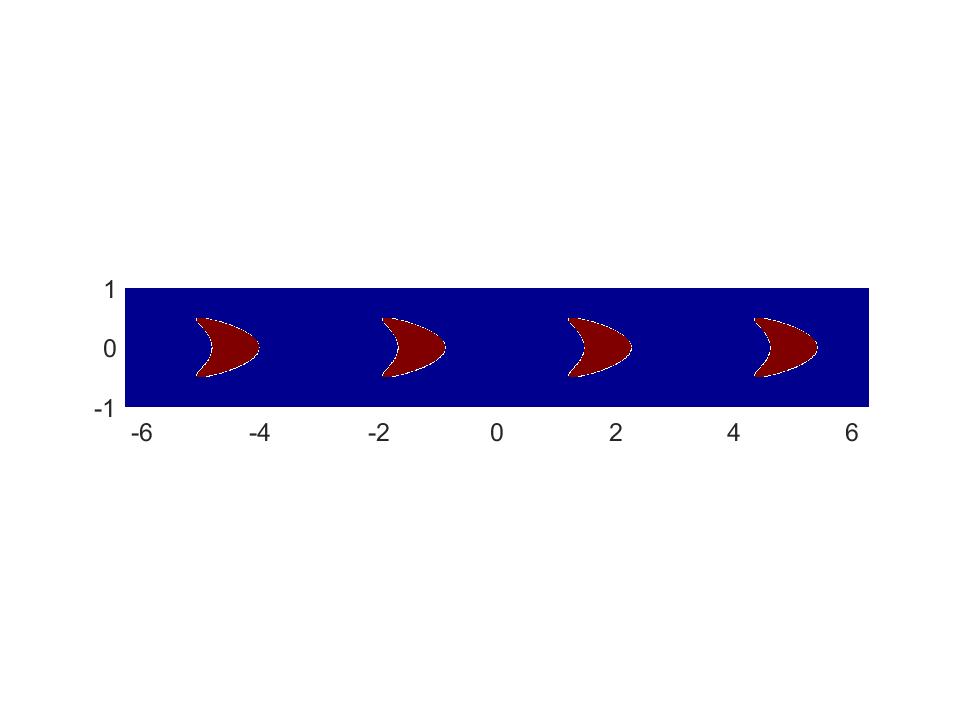}
\caption{True geometry in $(-2\pi,2\pi)$.}
\end{subfigure}

\begin{subfigure}{\textwidth}
\centering
\includegraphics[trim=0 250 0 250,clip,width=0.85\textwidth]{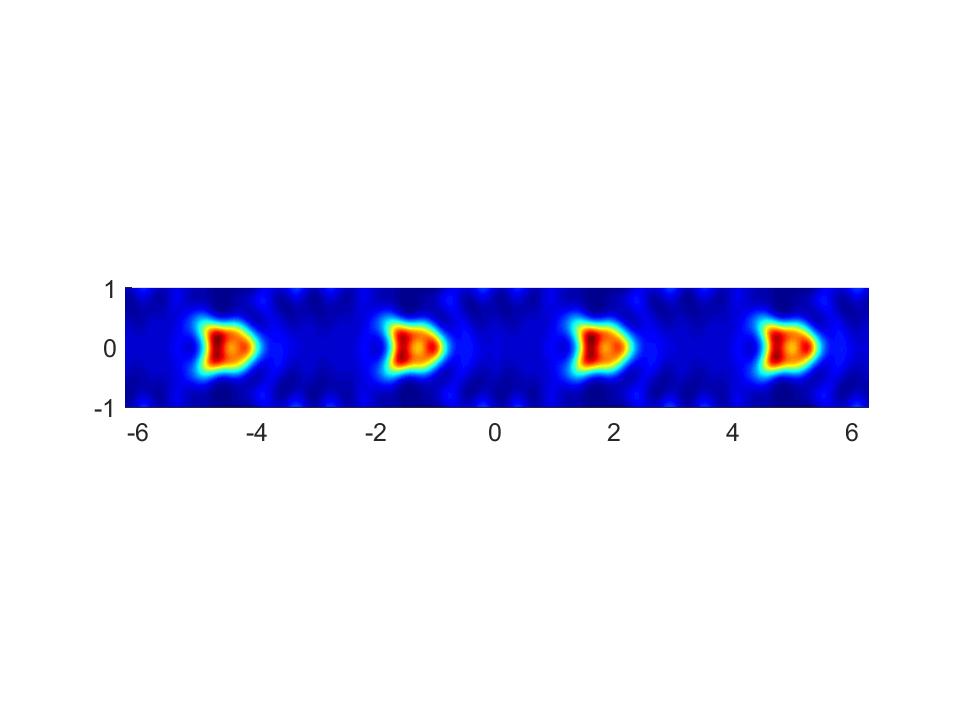}
\caption{Reconstruction by the proposed sampling method.}
\end{subfigure}

\begin{subfigure}{\textwidth}
 \centering
\includegraphics[trim=0 250 0 250,clip,width=0.85\textwidth]{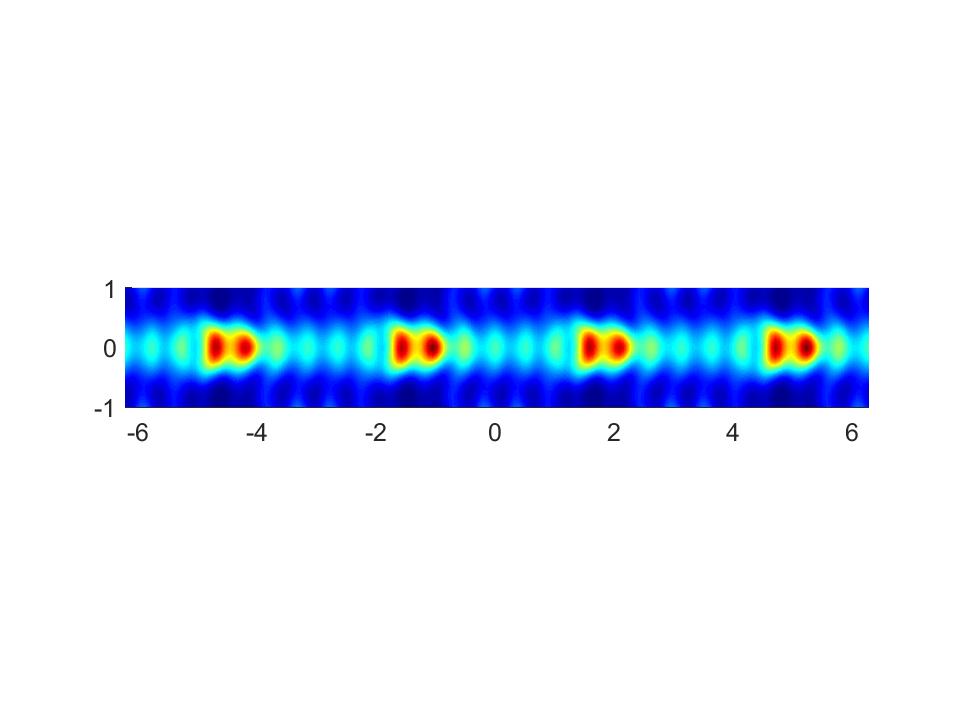}
\caption{Reconstruction by the orthogonality sampling method.}
\end{subfigure}

\begin{subfigure}{\textwidth}
\centering
\includegraphics[trim=0 250 0 250,clip,width=0.85\textwidth]{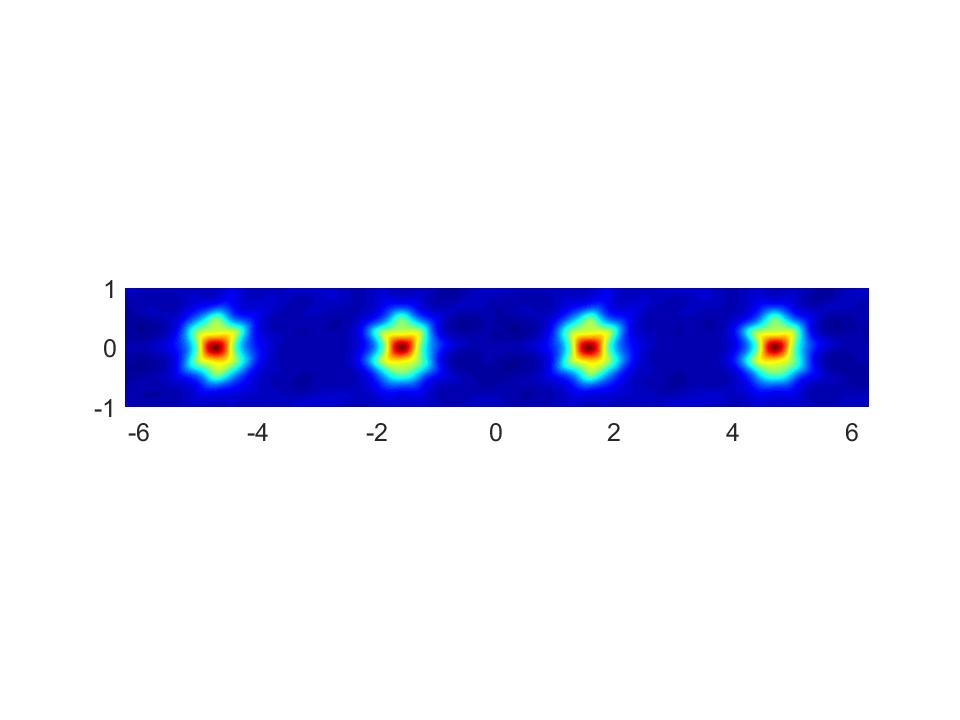}
\caption{Reconstruction by the factorization method.}
\end{subfigure}

\caption{Reconstruction of kite-shaped scatterers using different sampling methods.}
\label{fig9}

\end{figure}


\begin{figure}[h!]
\centering

\begin{subfigure}{\textwidth}
\centering
\includegraphics[trim=0 250 0 250,clip,width=0.85\textwidth]{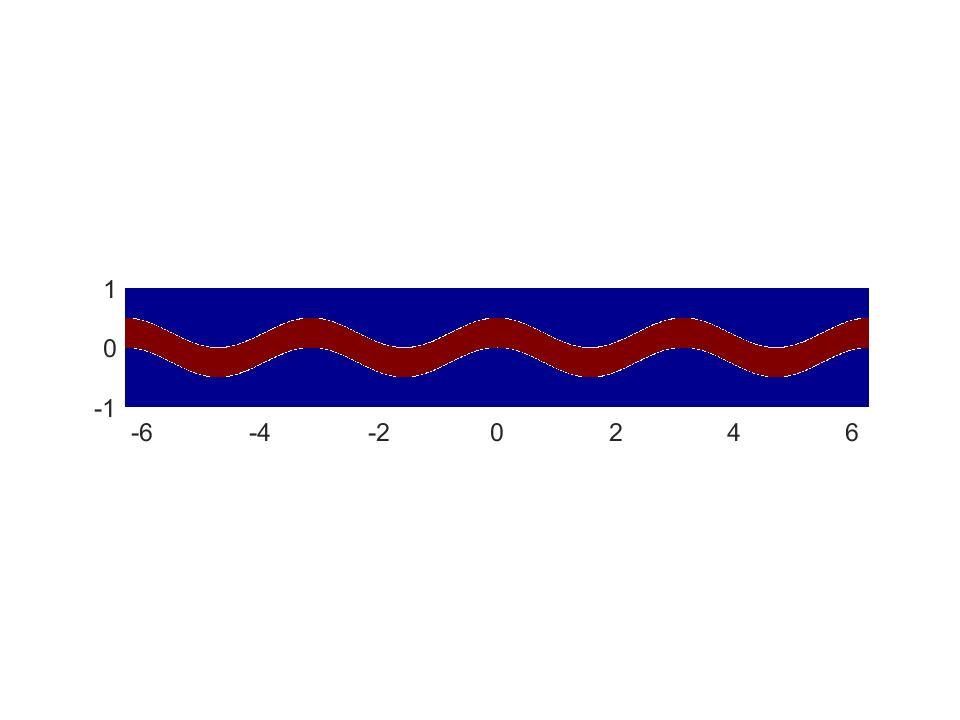}
\caption{True geometry in $(-2\pi,2\pi)$.}
\end{subfigure}

\begin{subfigure}{\textwidth}
\centering
\includegraphics[trim=0 250 0 250,clip,width=0.85\textwidth]{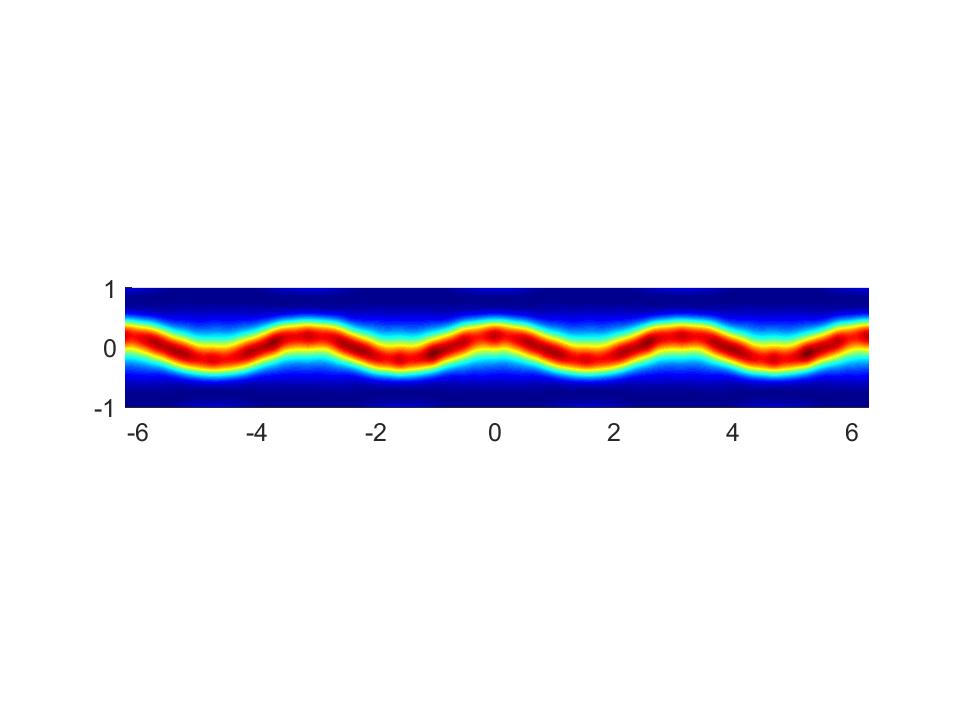}
\caption{Reconstruction by the proposed sampling method.}
\end{subfigure}

\begin{subfigure}{\textwidth}
\centering
\includegraphics[trim=0 250 0 250,clip,width=0.85\textwidth]{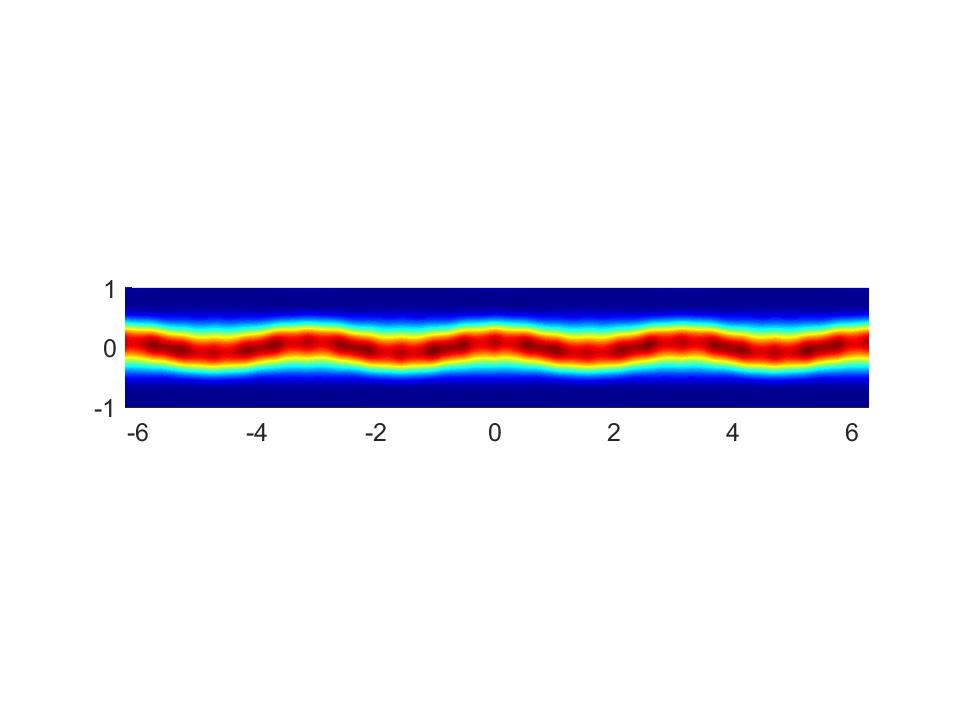}
\caption{Reconstruction by the orthogonality sampling method.}
\end{subfigure}

\begin{subfigure}{\textwidth}
\centering
\includegraphics[trim=0 250 0 250,clip,width=0.85\textwidth]{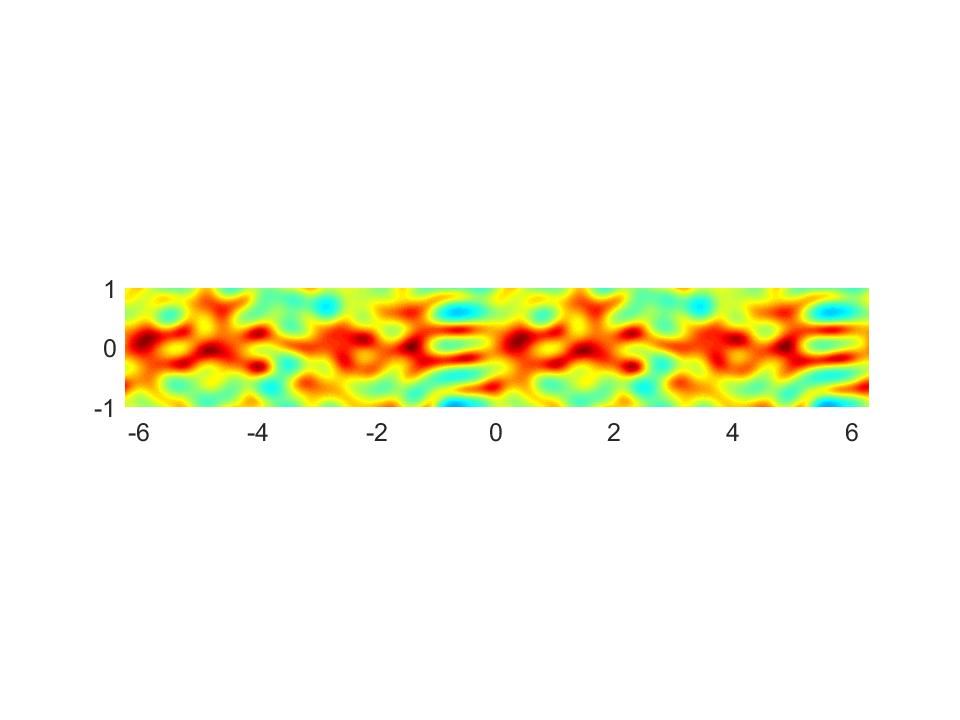}
\caption{Reconstruction by the factorization method.}
\end{subfigure}

\caption{Reconstruction of a sinusoidal scatterer using different sampling methods.}
\label{fig10}

\end{figure}


\vspace{0.5cm}

\textbf{Acknowledgement.} This work was partially supported by NSF grants DMS-1812693 and DMS-2208293.

\bibliographystyle{plain}
\bibliography{ip-biblio2}

\end{document}